\documentclass[a4paper]{amsart}
\usepackage{amssymb,amsmath,amsthm}
\usepackage{euscript}
\usepackage[neverdecrease]{paralist}
\usepackage{ifpdf} 
\ifpdf
	\usepackage[pdftex]{graphicx}
	\DeclareGraphicsExtensions{.pdf}
\usepackage[pdftex,plainpages=false,pdfpagelabels,linktocpage,pdfstartview=FitH,colorlinks=true,linkcolor=blue,citecolor=magenta]{hyperref}
\else
	\usepackage[dvips]{graphicx}
	\usepackage[hypertex,linkcolor=cyan]{hyperref}
\fi

\newcommand{\MySec}{\S}
\newcommand{\MySecs}{\S\S}

\theoremstyle{plain}
\newtheorem{thm}{Theorem}[section]
\newtheorem{cor}[thm]{Corollary}
\newtheorem{lemm}[thm]{Lemma}
\newtheorem{prop}[thm]{Proposition}
\theoremstyle{definition}
\newtheorem{defn}[thm]{Definition}
\theoremstyle{remark}
\newtheorem{rem}[thm]{Remark}

\numberwithin{equation}{section}

\newcommand{\D}{\mathrm{d}}

\newcommand{\diff}[2][]{\frac{\D #1}{\D #2}}
\newcommand{\diffat}[3][]{\left.\diff[#1]{#2}\right|_{#3}}

\DeclareMathOperator{\Stab}{Stab}
\DeclareMathOperator{\stab}{stab}
\DeclareMathOperator{\Ad}{Ad} \DeclareMathOperator{\ad}{ad}
 \DeclareMathOperator{\Aut}{Aut}
 \DeclareMathOperator{\rank}{rank}
\DeclareMathOperator{\trace}{trace}
\DeclareMathOperator{\Fix}{Fix}
\DeclareMathOperator{\Lag}{Lag}
\DeclareSymbolFont{script}{U}{eus}{m}{n}
\DeclareSymbolFontAlphabet{\mathscr}{script}
\DeclareMathSymbol{\EuWedge}{0}{script}{"5E}
\newcommand{\Wedge}{\EuWedge} \newcommand{\R}{\mathbb{R}}
\newcommand{\C}{\mathbb{C}} \newcommand{\Z}{\mathbb{Z}}
\newcommand{\K}{\mathbb{K}} \newcommand{\pr}{\mathbb{P}}
 \newcommand{\HH}{\mathbb{H}}
\newcommand{\cL}{\mathcal{L}} \newcommand{\cD}{\mathcal{D}}
\newcommand{\cN}{\mathcal{N}} \newcommand{\cG}{\mathcal{G}}
\newcommand{\fs}{\mathfrak{s}} \newcommand{\fso}{\mathfrak{so}}
\newcommand{\fsl}{\mathfrak{sl}} \newcommand{\fp}{\mathfrak{p}}
\newcommand{\fg}{\mathfrak{g}} \newcommand{\fh}{\mathfrak{h}}
\newcommand{\fq}{\mathfrak{q}} 
\newcommand{\fk}{\mathfrak{k}} \newcommand{\fm}{\mathfrak{m}}
\newcommand{\fa}{\mathfrak{a}} 
 \newcommand{\fr}{\mathfrak{r}}

\def\fc{\mathfrak{c}} 
 
\newcommand{\rO}{\mathrm{O}} \newcommand{\rSO}{\mathrm{SO}}
\newcommand{\rSL}{\mathrm{SL}} \newcommand{\rSp}{\mathrm{Sp}}
 \newcommand{\rSU}{\mathrm{SU}}
\newcommand{\rU}{\mathrm{U}} \newcommand{\rPSL}{\mathrm{PSL}}

\newcommand{\ttrans}{\EuScript{T}} \newcommand{\darb}{\EuScript{D}}

\newcommand{\ul}[1]{\underline{#1}}
\newcommand{\cl}[1]{\overline{#1}}  
\DeclareMathOperator{\im}{Im} \DeclareMathOperator{\cross}{cr}
\DeclareMathOperator{\dom}{dom} 
\newcommand{\Cl}{{C\ell}} 
\newcommand{\half}{\tfrac{1}{2}}
\newcommand{\ci}{\theta} 
\newcommand{\Span}[1]{\langle#1\rangle} \newcommand{\Dt}{\nabla^t}
\newcommand{\set}[1]{\{#1\}}
\newcommand{\act}{\,} 

\begin{document}
\title{Isothermic submanifolds of symmetric $R$-spaces} 
\author[F.E. Burstall]{Francis E. Burstall}
\address{Department of Mathematical Sciences\\ University of Bath\\
Bath BA2 7AY\\UK}
\email{feb@maths.bath.ac.uk}
\author[N.M. Donaldson]{Neil M. Donaldson}
\address{Department of Mathematics\\University of
California at Irvine\\Irvine, CA 92697-3875\\USA}
\email{ndonalds@math.uci.edu}
\author[F. Pedit]{Franz Pedit}
\address{Mathematisches Institut der  Universit{\"a}t T\"ubingen\\
  Auf der Morgenstelle 10\\
  72076 T\"ubingen\\GERMANY\\
  and
  Department of Mathematics \\
  University of Massachusetts\\
  Amherst, MA 01003, USA }
\email{pedit@mathematik.uni-tuebingen.de}
\author[U. Pinkall]{Ulrich Pinkall}
\address{Fachbereich Mathematik, MA 3-2\\
Technische Universit\"at Berlin\\
Strasse des 17. Juni 136\\
10623 Berlin\\
GERMANY}
\email{pinkall@math.tu-berlin.de}

\begin{abstract}
We extend the classical theory of isothermic surfaces in conformal
$3$-space, due to Bour, Christoffel, Darboux, Bianchi and others, to
the more general context of submanifolds of symmetric $R$-spaces with
essentially no loss of integrable structure.
\end{abstract}
\maketitle

\section*{Introduction}

\subsection*{Background}
\label{sec:background}

A surface in $\R^3$ is \emph{isothermic} if, away from umbilics, it
admits coordinates which are simultaneously conformal and curvature
line or, more invariantly, if it admits a holomorphic quadratic
differential $q$ which commutes with the (trace-free) second
fundamental form (then the coordinates are $z=x+iy$ for which $q=\D
z^{2}$).  Examples include surfaces of revolution, cones and
cylinders; quadrics and constant mean curvature surfaces (where
$q$ is the Hopf differential).

Starting with the work of Bour \cite[\S54]{Bour1862}, isothermic
surfaces were the focus of intensive study by the geometers of the
late 19th and early 20th centuries with contributions from
Christoffel, Cayley, Darboux, Demoulin, Bianchi, Calapso, Tzitz\'eica
and many others.  There has been a recent revival of interest in the
topic thanks to Cie{\'s}li{\'n}ski--Goldstein--Sym
\cite{Cieslinski1995} who pointed out the links with soliton theory:
indeed, isothermic surfaces constitute an integrable system with a
particularly beautiful and intricate transformation theory, some
aspects of which we list below:
\begin{asparadesc}
\item[Conformal invariance] Since the trace-free second fundamental
form is invariant (up to scale) under conformal diffeomorphisms of
$\R^{3}$, such diffeomorphisms preserve the class of isothermic
surfaces.  Thus isothermic surfaces are more properly to be viewed as
surfaces in the conformal $3$-sphere.
\item[Deformations]At least locally, isothermic surfaces admit a
$1$-parameter deformation preserving the conformal structure and
trace-free second fundamental form: this is the
\emph{$T$-transformation} of Bianchi \cite{Bianchi1905} and Calapso
\cite{Calapso1903}.  In fact, isothermic surfaces are characterised
by the existence of such a deformation \cite{Cartan1923} (see
\cite{Musso1995,Burstall2002b,Burstall2004c} for modern treatments).
\item[Darboux transformations] According to Darboux
\cite{Darboux1899}, given an isothermic surface, one may locally
construct a $4$-parameter family of new isothermic surfaces.
Analytically, this is accomplished by solving an integrable system of
linear differential equations; geometrically, the two surfaces
envelop a conformal Ribaucour sphere congruence\footnote{That is, the
two surfaces admit a common parametrisation under which both
conformal structures and curvature lines coincide and, through each
pair of corresponding points, there is a $2$-sphere tangent to both
surfaces.}.  These transformations, the \emph{Darboux
transformations}, are analogous to the B\"acklund transformations of
constant curvature surfaces (indeed, specialise to the latter in
certain circumstances \cite{Hertrich-Jeromin1997,Inoguchi2005}) and there is a
Bianchi permutability theorem \cite{Bianchi1905a} relating iterated
Darboux transforms.  There are additional permutability theorems relating
$T$-transforms and Darboux transforms \cite{Hertrich-Jeromin2003} and
also relating Darboux transforms with the Euclidean Christoffel
transform \cite{Christoffel1867} (see also \cite{Bour1862}) via
$T$-transforms \cite{Bianchi1905}.
\item[Curved flats] Curved flats were introduced by Ferus--Pedit
\cite{Ferus1996} and are an integrable system defined on
submanifolds of a symmetric space $G/H$ which, in non-degenerate
cases, coincides with the $G/H$-system of Terng \cite{Terng1997}.  In
particular, the space $S^3\times S^3\setminus\Delta$ of pairs of
distinct points in $S^3$ is a (pseudo-Riemannian) symmetric space for
the diagonal action of the conformal diffeomorphism group and curved
flats in this space are the same as pairs of isothermic surfaces
related by a Darboux transformation \cite{Burstall1997a}.
\item[Discrete theory] Bobenko--Pinkall \cite{BobenkoPinkall96a} show
that the combinatorics of iterated Darboux transforms give rise to an
integrable system on quad-graphs that is a convincing discretisation
of isothermic surfaces and shares with them all the classical
transformation properties listed above
\cite{JerominHoffmannPinkall1999,Jeromin2000,BobenkoSuris2007}.
\end{asparadesc}
For more information on isothermic surfaces, we refer the Reader to
Hertrich-Jeromin's encyclopedic monograph \cite{Hertrich-Jeromin2003}.

\subsection*{Manifesto}
\label{sec:manifesto}

In this paper, we will show that the preceding theory continues to
hold, in almost every detail, when we replace the conformal
$3$-sphere by an arbitrary \emph{symmetric $R$-space}.  There are a number
of approaches to symmetric $R$-spaces \cite{Takeuchi1965}: they
comprise Hermitian symmetric spaces and their real forms; they are
the compact Riemannian symmetric spaces which admit a Lie group of
diffeomorphisms strictly larger than their isometry group
\cite{Nagano1965} and, basic for us, they are the conjugacy classes
of parabolic subalgebras of a real semisimple Lie algebra with
abelian nilradicals.  Examples include projective spaces and
Grassmannians (real, complex and quaternionic) and quadrics (thus
conformal spheres of arbitrary signature).  The symmetric
$R$-spaces $G/P$, for $G$ simple, are enumerated in
Table~\ref{tab:r-space} on page~\pageref{tab:r-space}.

To see how an isothermic submanifold could be defined in this
setting, let us begin with the following manifestly conformally
invariant reformulation of the isothermic condition.  View the
holomorphic quadratic differential $q$ of an isothermic surface
$f:\Sigma\to S^{3}$ as a $T^{*}\Sigma$-valued $1$-form and thus, via
$\D f$, as a $f^{-1}T^{*}S^3$-valued $1$-form.  Now contemplate the
group $G$ of conformal diffeomorphisms of $S^{3}$, an open subgroup
of $\rO(4,1)$ and so semisimple, and its Lie algebra $\fg$.  Since
$G$ acts transitively on $S^3$, we have, for each $x\in S^{3}$, an
isomorphism $T_xS^3\cong\fg/\fp_{x}$, where $\fp_{x}$ is the
infinitesimal stabiliser of $x$, and so, via the Killing form, an
isomorphism $T_x^{*}S^3\cong\fp_x^{\perp}\subset\fg$.  Thus $q$ may be
identified with a certain $\fg$-valued $1$-form $\eta$.  The key
observation now is that $q$ is a holomorphic quadratic differential
commuting with the second fundamental form of $f$ if and only if
$\eta$ is a \emph{closed} $1$-form \cite{Burstall2004c}.

The decisive property enjoyed by the conformal sphere is that each
$\fp_{x}^{\perp}$ is an abelian subalgebra of $\fg$ from which it
quickly follows that each of the connections $\D+t\eta$, $t\in\R$, is
flat.  This provides a zero-curvature formulation of isothermic
surfaces from which the whole theory may ultimately be deduced.

However, according to Grothendieck \cite{Grothendieck1957},the
condition that an infinitesimal stabiliser $\fp_{x}$ has abelian
Killing polar is precisely that $\fp_{x}$ is a parabolic subalgebra
with abelian nilradical $\fp_{x}^{\perp}$.  The map $x\mapsto\fp_{x}$
then identifies $S^3$ with a conjugacy class of such parabolic
subalgebras. 

It is now clear that this formulation of the isothermic condition
makes sense for maps into any symmetric $R$-space.  For a semisimple Lie
algebra $\fg$ and conjugacy class $N$ of parabolic subalgebras
$\fp<\fg$ with abelian nilradicals, we may view $T_{\fp}^{*}N$ as an
abelian subalgebra $\fp^{\perp}<\fg$ and say that a map $f:\Sigma\to N$ is
isothermic if there is an $f^{-1}T^{*}N$-valued $1$-form $\eta$ which
is closed as an element of $\Omega_{\Sigma}^1\otimes\fg$.  We
immediately get a pencil of flat connections $\D+t\eta$ and, as we
will show, the theory of isothermic maps can be developed in complete
analogy with the conformal theory sketched above.

\subsection*{Road-map}
\label{sec:road-map}

To orient the Reader, we briefly sketch the contents of the paper.

The first two sections are preparatory in nature.  We recall in
\S\ref{sec:parab} the basic definitions and facts about parabolic
subalgebras $\fp$ of a real semisimple Lie algebra $\fg$.  Particular
emphasis is given to the notion of parabolic subalgebras
\emph{complementary} to a given parabolic $\fp$.  Then, in
\S\ref{sec:r-spaces-symmetric}, we rehearse what we need of the
theory of $R$-spaces and symmetric $R$-spaces viewed as conjugacy
classes of parabolics.  In particular, we note that such spaces have
convenient charts generalising those given by stereoprojection on
spheres.  Complementary parabolic subalgebras arise in two different
ways.  Firstly, the set of parabolic subalgebras complementary to
some element of a conjugacy class $M$ is itself a conjugacy class
$M^{*}$ which we call the \emph{dual} $R$-space.  It may happen that
$M=M^{*}$ (conformal spheres are an example) and then we say that $M$
is \emph{self-dual}.  Another, non-self-dual, example of this duality
is given by the dual projective spaces of lines and hyperplanes in a
vector space.  Secondly, the set $Z\subset M\times M^{*}$ of
pair-wise complementary parabolic subalgebras is shown to be a
pseudo-Riemannian symmetric space (in fact, a \emph{para-Hermitian}
symmetric space as studied by Kaneyuki
\cite{Kaneyuki1985,Kaneyuki1987}).  This symmetric space plays the
role of $S^3\times S^{3}\setminus\Delta$ in the conformal theory.

With these preliminaries out of the way, in \S\ref{sec:isoth-subm} we
define isothermic maps into symmetric $R$-spaces and explore their
transformation properties.  Just as in the conformal case, we can
define $T$-transforms, Darboux transforms and Christoffel transforms
and prove some permutability theorems.  We also show how Christoffel
transforms are limits of Darboux transforms: a result which may have
some novelty even in the classical context.

Analogues of the classical Bianchi Permutability Theorem for Darboux
transforms are explored in \S\ref{sec:bianchi-perm-self}.  Here we
must restrict attention to self-dual targets but, in that setting,
find that the classical results carry through in every detail.  An
important role is played in this analysis by a $G$-invariant family
of circles in a self-dual symmetric $R$-space which share many
properties of circles in a conformal sphere: they are determined by three
generic points and carry an invariant projective structure (and
thus a cross-ratio).  We then find that four isothermic maps in the
configuration of the permutability theorem have corresponding points
concircular with prescribed constant cross-ratio (a theorem of
Demoulin \cite{Demoulin1910} in the classical case).  Our methods
here seem to be more efficient than those known in the conformal case
and yield a version of Bianchi's cube theorem with almost no extra
work.  A side-benefit of our approach is an integrable theory of
\emph{discrete} isothermic nets in self-dual symmetric $R$-spaces
which we sketch in \S\ref{sec:discr-isoth-surf}.  The cube theorem in
this context amounts to 3D consistency of our discrete integrable
system in the sense of Bobenko--Suris \cite{Bobenko2002}.

Curved flats are considered in \S\ref{sec:curved}: a point-wise
complementary pair of isothermic maps define a map
$\phi=(f,\hat{f}):\Sigma\to Z\subset M\times M^{*}$ and, in strict
analogy with the conformal case \cite{Burstall1997a}, $\phi$ is a
curved flat if and only if $f$ and $\hat f$ are Darboux transforms of
each other.  In addition, we consider the dressing transforms
\cite{Bruck2002,Burstall2004} of curved flats and prove that, in the
self-dual case, the four isothermic maps comprising a curved flat and
its dressing transform are a Bianchi quadrilateral of isothermic maps
related by Darboux transforms.  This last appears to be a new result
even in the conformal case.

Finally, in \S\ref{sec:examples}, we show that the $1$-form $\eta$
defines a quadratic differential on $\Sigma$ and consider the case
where this is non-degenerate (in particular, such isothermic maps
immerse).  We show that there is a sharp Lie-theoretic upper bound on
the dimension of these non-degenerate isothermic submanifolds and so,
in particular, settle the existence question for isothermic maps.

Isothermic submanifolds of maximal dimension in $\R P^n$ turn out to
be curves and we contemplate the simplest case of isothermic curves
in $\R P^1$.  These are the same as parametrised curves but, perhaps
surprisingly, our theory still has something to offer: at the level
of curvatures, the relation between a curved flat and its constituent
isothermic maps amounts to the Miura transform and we find an
integrable dynamics of isothermic curves, commuting with $T$ and
Darboux transforms.  Passing to curvatures, the dynamics of
isothermic curves gives the KdV equation, that of curved flats the
mKdV equation while the Darboux transform extends to yield the
Wahlquist--Estabrook B\"acklund transform
\cite{WahlquistEstabrook1971} of the KdV equation.

\subsection*{Acknowledgements}

During the preparation of this work, we have derived great benefit
from conversations with David Calderbank; Udo Hertrich-Jeromin; Nigel
Hitchin; Soji Kaneyuki; Katrin Leschke and Chuu-Lian Terng.  Special
thanks are due from the first author to Yoshihiro Ohnita who first
taught him about symmetric $R$-spaces over green tea ice cream in 1998.

\section{Algebraic preliminaries}\label{sec:parab}

An $R$-space is a conjugacy class of parabolic subalgebras in a
non-compact semisimple Lie algebra.  We begin by collecting some
elements of the underlying algebra.

Here, and throughout this paper, $\fg$ will be a real non-compact
semisimple Lie algebra with adjoint group $G$. We denote the adjoint
action of $G$ on $\fg$ as well as the induced action on subalgebras
of $\fg$ by juxtaposition: thus, for $g\in G$, $\zeta\in\fg$, we
write $g\act\zeta$ for $\Ad(g)\zeta$.

\subsection{Parabolic subalgebras}
\label{sec:parab-subalg}

A subalgebra $\fp\leq\fg$ is parabolic if its complexification
$\fp^\C\leq\fg^\C$ is parabolic in the sense that it contains a Borel
(thus maximal solvable) subalgebra of $\fg^\C$.  However, it will be
useful for us to adopt the following equivalent definition
\cite[Proposition~4.2]{Burstall1990} (see also \cite[Lemma
4.2]{Grothendieck1957} and \cite[Definition~2.1]{Calderbank2005}):
\begin{defn}\label{defn:para} A subalgebra $\fp$ of a non-compact
semisimple Lie algebra $\fg$ is \emph{parabolic} if and only if the
polar $\fp^\perp$ with respect to the Killing form is a nilpotent
subalgebra. 

In this case, $\fp^\perp$ is the nilradical of $\fp$ (in particular,
$\fp^\perp<\fp$) and we say that $\fp$ has \emph{height} $n$ if
$\fp^\perp$ is $n$-step nilpotent.
\end{defn}
In particular, $\fp$ has height one if and only if $\fp^\perp$ is
abelian. 

There are a finite number of conjugacy classes of parabolic
subalgebras of $\fg$.  If $\fg$ is complex\footnote{Real parabolic
subalgebras of a complex $\fg$ are in fact complex: indeed,
\eqref{eq:graded} gives a decomposition into complex subspaces.},
they are in bijective correspondence with subsets of nodes (the
\emph{crossed nodes}) of the Dynkin diagram of $\fg$.  If, in
addition, $\fg$ is simple, one computes the height of a parabolic
subalgebra by summing the coefficients in the highest root of the
simple roots corresponding to crossed nodes (see, for example,
\cite{Baston1989}).  If $\fg$ is semisimple, any parabolic subalgebra
is a direct sum of parabolic subalgebras (possibly not
proper\footnote{Thus height zero.}!) of simple ideals and the height
is the maximum of the heights of these.  If $\fg$ is not complex, the
same analysis is available if Dynkin diagrams are replaced by Satake
diagrams with the caveat that the crossed nodes must all be white and
any node joined by an arrow to a crossed node must also be crossed
\cite[Th\'eor\`eme~3.1]{Matsumoto1964}.

In particular, height one parabolic subalgebras correspond to simple
roots with coefficient one in the highest root and if $\fg$ is not
complex, the corresponding node on the Satake diagram must be white
with no arrows.  It is now an easy matter to arrive at the
classification of conjugacy classes given in Table~\ref{tab:r-space}
on page~\pageref{tab:r-space}.   For example, there are no height one
parabolic subalgebras in any $\fg_2$, $\mathfrak{f}_4$ or
$\mathfrak{e}_8$ (no coefficient one simple roots) and none in
$\mathfrak{su}(p,q)$ for $p\neq q$ (no white nodes without arrows).

\subsection{Complementary subalgebras}
\label{sec:compl-subalg}

A parabolic subalgebra $\fp$ of height $n$ makes $\fg$ into a
filtered algebra: define inductively
\begin{equation*}
\fp^{(0)}=\fp,\ \fp^{(-1)}=\fp^\perp,\ \fp^{(j)}=
\begin{cases}
[\fp^\perp,\fp^{(j+1)}]&j\le -2,\\ 
{\fp^{(-1-j)}}^\perp&j\ge 1. 
\end{cases}
\end{equation*}
and observe that 
\begin{equation*}
\fg=\fp^{(n)}\supsetneq\cdots\supsetneq\fp^{(1)}\supsetneq\fp\supsetneq 
\fp^{(-1)}\supsetneq\cdots\supsetneq\fp^{(-n)}\supsetneq\fp^{(-n-1)}=\{0\}
\end{equation*}
while $[\fp^{(j)},\fp^{(k)}]\subset\fp^{(j+k)}$. 

\begin{defn}
Two parabolic subalgebras $\fp,\fq$ of height $n$ are
\emph{complementary} if
\begin{equation*}
  \fg=\fp^{(j)}\oplus\fq^{(-1-j)},
\end{equation*}
for all $j=-n,\ldots,n$. 
\end{defn}
Any non-trivial parabolic subalgebra admits complementary subalgebras:
\begin{lemm}[{\cite[Lemma~2.2]{Calderbank2005}}]
For any Cartan involution $\ci$ of $\fg$ and any parabolic $\fp<\fg$,
$\fp$ and $\ci\fp$ are complementary. 
\end{lemm}
Moreover, the complementary subalgebras to a fixed $\fp$ are all
conjugate:
\begin{lemm}[{\cite[Lemma~2.5]{Calderbank2005}}]\label{lemm:simptrans}
$\exp\fp^\perp<G$ acts simply transitively on the set of parabolic
subalgebras complementary to $\fp$. 
\end{lemm}
Since $\fp^\perp$ is nilpotent,
$\exp|_{\fp^\perp}:\fp^\perp\to\exp\fp^\perp$ is a diffeomorphism so
that Lemma~\ref{lemm:simptrans} tells us that this set of
complementary parabolics is an affine space modelled on $\fp^\perp$. 

A complementary pair $(\fp,\fq)$ of height $n$ parabolic subalgebras
splits the filtered structure of $\fg$ into a graded one: set
$\fg_j:=\fp^{(j)}\cap\fq^{(-j)}$
and then $\fg=\bigoplus_{j=-n}^n\fg_j$ while
$[\fg_j,\fg_k]\subset\fg_{j+k}$. 
Thus the map given by $\zeta\mapsto j \zeta$, for $\zeta\in\fg_j$ is
a derivation and so, since $\fg$ is semisimple, given by the adjoint
action of a unique element $\xi$ in the centre of $\fg_0$.  We call
$\xi=\xi_\fp^\fq$ the \emph{grading} or \emph{canonical element} of the pair
$(\fp,\fq)$. 

Of course, we can recover the complementary pair from the grading
element: 
\begin{align}\label{eq:graded}
\fp&=\bigoplus_{j=-n}^0\fg_j &
\fq&=\bigoplus_{j=0}^n\fg_j. 
\end{align}

\begin{rem}\label{h:comp-ht-1}
When $\fp$ has height $1$, the theory simplifies considerably.  In
particular, $\fq$ is complementary to $\fp$ if and only if
$\fp^\perp\cap\fq^\perp=\{0\}$ if and only if $\fg=\fp\oplus\fq^\perp$.
\end{rem}

As a simple application of the existence of grading elements, we see
that parabolic subalgebras are self-normalising:
\begin{lemm}
\label{lem:normal}
For $\zeta\in\fg$, $[\zeta,\fp]\subset\fp$ if and only if $\zeta\in\fp$. 
\end{lemm}
\begin{proof}
We know that $\ad\xi$ is invertible on $\fp^\perp$ so that
$[\xi,\fp^\perp]=\fp^\perp$ giving $[\fp,\fp^\perp]=\fp^\perp$.  Now,
if $[\zeta,\fp]\subset\fp$ then $[\zeta,\fp]$ is Killing orthogonal
to $\fp^\perp$ or, equivalently, $\zeta\in [\fp,\fp^\perp]^\perp=\fp$. 
\end{proof}

\section{$R$-spaces and symmetric $R$-spaces}
\label{sec:r-spaces-symmetric}

\subsection{$R$-spaces}
\label{sec:r-spaces}

\begin{defn}[\cite{Tits1954}]
An \emph{$R$-space}  of \emph{height $n$} is a
conjugacy class of height $n$ parabolic subalgebras. 
\end{defn}

Let $M$ be an $R$-space.  Thus $M$ is a homogeneous space of $G$ with
stabilisers which are, by definition, \emph{parabolic subgroups} of
$G$.  By Lemma~\ref{lem:normal}, the stabiliser $P$ of $\fp\in M$ has
Lie algebra $\fp$.  

Now let $K$ be a maximal compact subgroup of $G$.  By virtue of the
Iwasawa decomposition of $G$, $K$ acts transitively on $M$
\cite[Theorem~7]{Takeuchi1965} and so, in particular, $M$ is compact. 

Let $\fp_0\in M$ and fix a complementary subalgebra $\fp_\infty$.  Set 
\begin{equation*}
\Omega_{\fp_\infty}=\{\fp\in M: \text{$\fp$ is complementary to
$\fp_\infty$}\}\subset M. 
\end{equation*}
Then $\Omega_{\fp_\infty}$ is an open, dense neighbourhood of $\fp_0$
in $M$: it is the ``big cell'' in the cell decomposition of $M$
induced by the Bruhat decomposition of $G$
\cite[Theorem~8]{Takeuchi1965}.  Moreover, from
Lemma~\ref{lemm:simptrans}, we see that the map
\begin{equation}
\label{eq:1}
\zeta\mapsto
\exp\zeta\,\fp_0:\fp_\infty^\perp\to\Omega_{\fp_\infty}
\end{equation}
is a diffeomorphism.  We call this map
\emph{inverse-stereoprojection} with respect to $(\fp_0,\fp_\infty)$.

If $\fg$ is a complex Lie algebra, the $R$-spaces are the well-known
generalised flag manifolds: these are homogeneous complex (in fact,
projective, algebraic) manifolds with a $K$-invariant K\"ahler
structure. The remaining $R$-spaces are real forms of generalised
flag manifolds. Indeed, let $\fg$ be non-complex and $M$ an $R$-space
for $G$. For $\fp\in M$, $\fp^\C$ is a parabolic subalgebra of
$\fg^\C$ of the same height as $\fp$ and we denote by $M^\C$ the
$G^\C$-conjugacy class of the $\fp^\C$.  The map $\fp\mapsto\fp^\C$
embeds $M$ in $M^\C$ as the fixed set of an anti-holomorphic
involution of $M^\C$: indeed, the conjugation $\sigma$ of $\fg^\C$
across $\fg$ induces such an involution $\sigma:\fp\mapsto\sigma\fp$
on $M^\C$ with $M\subset\Fix(\sigma)$ and equality in this last
inclusion follows from Matsumoto's classification result 
\cite[Th\'eor\`eme~3.1]{Matsumoto1964} since there is at most
one real conjugacy class of parabolic subalgebras in any complex
one (see \cite[Theorem~10]{Takeuchi1965} for an alternative argument).

\subsection{Symmetric $R$-spaces: definition and examples}
\label{sec:symmetric-r-spaces}

The following definition is due to Takeuchi \cite{Takeuchi1965}. 
\begin{defn}
A \emph{symmetric $R$-space} is an $R$-space of height 1. 
\end{defn}

Symmetric $R$-spaces are distinguished by the fact that they are
Riemannian symmetric spaces for the maximal compact subgroup $K$.
Indeed, with $\ci$ the corresponding Cartan involution of $\fg$, the
involution at $\fp\in M$ is given by
$\tau_\fp=\exp(i\pi\xi)\in\Aut(\fg)$ where $\xi$ is the grading
element for the complementary pair $(\fp,\ci\fp)$.

Thus symmetric $R$-spaces are compact Riemannian symmetric spaces
admitting a Lie group of diffeomorphisms $G$ properly containing the
isometry group.  It is a celebrated theorem of Nagano
\cite[Theorem~3.1]{Nagano1965} that, up to covers, these are
essentially the only compact Riemannian symmetric spaces admitting
such a group. 

However, in what follows, we will be exclusively interested in the
$G$-invariant geometry of symmetric $R$-spaces and shall
relegate their Riemannian structure (which requires the additional
choice of a maximal compact subgroup $K$) firmly to the background
(see, however, \MySec\ref{sec:dimension-bounds} and
\MySec\ref{sec:curv-flats-riem}).

For $\fg$ complex, the symmetric $R$-spaces $M$ are precisely the
Hermitian symmetric spaces of compact type.  Here $G$ is the
group of biholomorphisms.  As we saw in \MySec\ref{sec:r-spaces},
any symmetric $R$-space for non-complex $\fg$ is a real form of such
a Hermitian symmetric space. 

Let us now illustrate some of the above theory with some concrete
examples:

\subsubsection{Example: the conformal $n$-sphere}
\label{sec:example:-conformal-n}

Contemplate the Lorentzian vector space $\R^{n+1,1}$; an
$(n+2)$-dimensional vector space with an inner product $(\,,\,)$ of
signature $(n+1,1)$\footnote{Thus $n+1$ positive directions and $1$
negative direction.}.  We distinguish the light-cone
$\cL=\{v\in\R^{n+1,1}_\times\colon (v,v)=0\}$ and consider its
projectivisation $\pr(\cL)\subset\pr(\R^{n+1,1})$ on which
$G=\Ad\rSO(n+1,1)\cong\rSO_0(n+1,1)$ acts transitively. 

There is a $G$-invariant conformal structure on $\pr(\cL)$: view a
section $s$ of the principal bundle $\cL\to\pr(\cL)$ as a map
into $\R^{n+1,1}$ and set $g_s=s^*(\,,\,)$.  Thus
\begin{equation*}
g_s(X,Y)=(\D_Xs,\D_Ys). 
\end{equation*}
It is not difficult to see that $g_s$ is positive definite and
that, for any $u:\pr(\cL)\to\R$, $g_{e^us}=e^{2u}g_s$ so
that we have a well-defined conformal structure. 

Let us identify the infinitesimal stabiliser
$\stab(\Lambda)\leq\fg=\fso(n+1,1)$ of a null line
$\Lambda\in\pr(\cL)$. First $\Lambda$ defines a flag
$\Lambda\subset\Lambda^\perp\subset\R^{n+1,1}$ and then, under the
canonical isomorphism
\begin{gather}\label{eq:orthiso}
{\Wedge}^2\R^{n+1,1}\cong\fso(n+1,1):u\wedge v\mapsto (u,\cdot)v-(v,\cdot)u,
\end{gather}
we readily see that
\begin{equation*}
\stab(\Lambda)=\Lambda\wedge\R^{n+1,1}+{\Wedge}^2\Lambda^\perp
\end{equation*}
The Killing polar of $\stab(\Lambda)$ is then the abelian subalgebra
\begin{equation*}
\stab(\Lambda)^\perp=\Lambda\wedge\Lambda^\perp. 
\end{equation*}
Thus $\stab(\Lambda)$ is a parabolic subalgebra of $\fg$ and the map
$\Lambda\mapsto\stab(\Lambda)$ is a $G$-isomorphism from $\pr(\cL)$
to the conjugacy class of $\stab(\Lambda)$ which is a symmetric
$R$-space. 

The general theory now tells us that $\pr(\cL)$ is a Riemannian
symmetric $K$-space for $K\cong\rSO(n+1)$ a maximal compact subgroup
of $G$.  In fact, $\pr(\cL)\cong S^n$ as $K$-spaces.  Indeed, a
choice of maximal compact amounts to a choice of unit time-like
vector $t_0\in\R^{n+1,1}$.  Then $K$ is the identity component of
$\Stab(t_0)\leq\rSO(n+1,1)$ and we identify $\pr(\cL)$ with the unit
sphere $S^n\subset\langle t_0\rangle^\perp$ via:
\begin{equation*}
x\mapsto \langle x+t_0\rangle:S^n\to\pr(\cL). 
\end{equation*}
For $s$ the unique section of $\cL$ with $(s,t_0)\equiv
-1$, this map is an isometry
$(S^n,g_{\mathrm{can}})\to\bigl(\pr(\cL),g_s\bigr)$ so that we
conclude that $\pr(\cL)$ is conformally equivalent to $S^n$. 
Moreover, this equivalence identifies $G$ with the group of
orientation-preserving conformal diffeomorphisms of $S^n$

This model of the conformal $n$-sphere is due to Darboux
\cite[Chapitre~VI]{Darboux1972}. 

For $\Lambda\in\pr(\cL)$, the parabolic subalgebras of $\fg$
complementary to $\stab(\Lambda)$ are precisely the stabilisers
$\stab(\hat\Lambda)$ for
$\hat\Lambda\in\pr(\cL)\setminus\{\Lambda\}$.  Indeed, given such a
$\hat\Lambda$, set $W=(\Lambda\oplus\hat\Lambda)^\perp$ to get a
decomposition 
\begin{equation*}
\R^{n+1,1}=\Lambda\oplus W\oplus\hat\Lambda
\end{equation*}
with $\Lambda\oplus W=\Lambda^\perp$,
$\hat\Lambda^\perp=W\oplus\hat\Lambda$. Thus
$\stab(\Lambda)^\perp=\Lambda\wedge W$ and
$\stab{\hat\Lambda}^\perp=\hat\Lambda\wedge W$ and these have zero
intersection whence $\stab(\Lambda)$ and $\stab(\hat\Lambda)$ are
complementary (see Remark~\ref{h:comp-ht-1}).  The corresponding
grading element has eigenvalues $-1$, $0$, $1$ on $\Lambda$, $W$,
$\hat\Lambda$ respectively.  Conversely, if $\xi$ is the grading
element of a complementary pair $(\stab(\Lambda),\fq)$, then
$\fq=\stab(\hat\Lambda)$ where $\hat\Lambda$ is the $+1$-eigenspace
for the action of $\xi$ on $\R^{n+1,1}$.  

In particular, we note that the parabolic subalgebras complementary
to some $\stab(\Lambda)$ lie in a single conjugacy class which
coincides with that of $\stab(\Lambda)$.

The inverse-stereoprojection introduced in \MySec\ref{sec:r-spaces}
coincides with the classical notion: choose
$\Lambda_0\neq\Lambda_\infty\in\pr(\cL)$ with infinitesimal
stabilisers $\fp_0,\fp_\infty$ and set
$\R^n=(\Lambda_0\oplus\Lambda_\infty)^\perp$.  Then
$\Omega_{\fp_\infty}=\pr(\cL)\setminus\{\Lambda_\infty\}$ and
\eqref{eq:1} reads:
\begin{equation*}
\zeta\mapsto\exp(\zeta)\Lambda_0:
\Lambda_\infty\wedge\R^n\to\pr(\cL)\setminus\{\Lambda_\infty\}. 
\end{equation*}
If we now choose $v_0\in\Lambda_0$, $v_\infty\in\Lambda_\infty$ with
$(v_0,v_\infty)=-1$, we can identify $\Lambda_\infty\wedge\R^n$ with
$\R^n$ via $x\wedge v_\infty \mapsto x=(x\wedge v_\infty)v_0$ and
then the diffeomorphism is
\begin{equation*}
x\mapsto\langle\exp(x\wedge v_\infty)v_0\rangle=\langle v_0+x+\half(x,x)v_\infty\rangle
\end{equation*}
which last is the expression for classical inverse-stereoprojection
that can be found in \cite{Bryant1984}. 

One can perform a similar analysis of the projective lightcone
$\pr(\cL^{p+1,q+1})$ of $\R^{p+1,q+1}$ to get a symmetric $R$-space
diffeomorphic to $(S^q\times S^q)/\Z_2$ which is the conformal
compactification via \eqref{eq:1} of $\R^{p,q}$. 

\subsubsection{Example: Grassmannians}
\label{sec:grassmannians}

Contemplate the Grassmannian $G_k(\C^n)$ of $k$-dimensional linear
subspaces of $\C^n$ on which $G=\rPSL(n,\C)$ acts transitively (thus,
for $k=1$, we are studying the projective geometry of $\C\pr^n$).  The
infinitesimal stabiliser of $W\in G_k(\C^n)$ has Killing polar
$\stab(W)^\perp=\hom(\C^n/W,W)$ which is an abelian subalgebra of
$\fsl(n,\C)$ so that the conjugacy class $M$ of these stabilisers is a
symmetric $R$-space.  The map $W\mapsto\stab(W):G_k(\C^n)\to M$ is
now an isomorphism of $G$-spaces. 

For $W\in G_k(\C^n)$, the parabolic subalgebras complementary to
$\stab(W)$ are all of the form $\stab(\hat W)$ for $\hat W\in
G_{n-k}(\C^n)$ with $W\oplus\hat W=\C^n$.  The grading
element of the pair $(W,\hat W)$ has eigenvalues $(k-n)/n$, $k/n$ on
$W$, $\hat W$ respectively. Again we see that the set of parabolic
subalgebras complementary to some $\stab(W)$, $W\in G_{k}(\C^n)$,
lies in a single conjugacy class but, in contrast to the case of
spheres, this conjugacy class does not coincide with that of the
$\stab(W)$ unless $2k=n$.

For Grassmannians, stereoprojection coincides with the usual affine
chart: given complementary $W_0\in G_k(\C^n)$ and $W_\infty\in
G_{n-k}(\C^n)$ with infinitesimal stabilisers $\fp_0,\fp_\infty$, we
identify $\Omega_{\fp_\infty}$ with $\bigl\{W\in G_k(\C^n):
W\cap W_\infty=\{0\}\bigr\}$, $\C^n/W_\infty$ with $W_0$ and thus
$\fp_\infty^\perp$ with $\hom(W_0,W_\infty)$.  Now \eqref{eq:1} reads
\begin{equation*}
T\mapsto (1+T)W_0: \hom(W_0,W_\infty)\cong \Omega_{\fp_\infty}. 
\end{equation*}

A similar analysis is available for the real Grassmannians
$G_k(\R^n)$ with $G=\rPSL(n,\R)$ and the quaternionic Grassmannians
$G_k(\HH^n)$ with $G=\mathrm{PSU}^*(2n)$.  These symmetric $R$-spaces
are real forms of the complex Grassmannians: indeed, they may be
viewed as the fixed set of the anti-holomorphic involution $W\mapsto
jW$ on $G_k(\C^n)$ induced by a real or quaternionic structure $j$ on
$\C^{n}$ (thus $j$ is an antilinear endomorphism of $\C^{n}$ with
$j=\pm1$).

There is one other real form of $G_{n}(\C^{2n})$: this is the set
$G^{\perp}_{n}(\C^{n,n})$ of maximal isotropic subspaces for a Hermitian
inner product of signature $(n,n)$ and so is the fixed set of the
anti-holomorphic involution $W\mapsto W^{\perp}$.

\subsubsection{Example: isotropic Grassmannians}
\label{sec:exampl-isotr-grassm}

Equip $\C^{2n}$ with a nondegenerate complex bilinear inner product
$q$ and contemplate the set of maximal isotropic subspaces of
$\C^{2n}$.  The group $\Ad\rSO(2n,\C)$ acts on these and there are
two orbits: for $W\subset\C^{2n}$ maximal isotropic, $\bigwedge^{n}W$
lies in an eigenspace of the Hodge $*$-operator and so we have two
orbits $J^{+}(\C^{2n})$ and $J^{-}(\C^{2n})$ according the sign of the
eigenvalue.

Under the isomorphism \eqref{eq:orthiso}, the infinitesimal
stabiliser of $W\in J^{\pm}(\C^{2n})$ is $\C^{2n}\wedge W$ with
Killing polar $\bigwedge^{2}W$, an abelian subalgebra of
$\fso(2n,\C)$.  Thus the conjugacy class $M^{\pm}$ of $\stab(W)$ is a symmetric
$R$-space and $W\mapsto\stab(W):J^{\pm}(\C^{2n})\to M^{\pm}$ is an
isomorphism of $G$-spaces.

The parabolic subalgebras complementary to $\stab(W)$ are those of
the form $\stab(\hat{W})$ where $\hat W\in J^{\pm}(\C^{2n})$ and
$\C^{2n}=W\oplus\hat W$.  Now, for $W^{\pm}\in J^{\pm}(\C^{2n})$, one
uses the characterisation of the orbits via the $*$-operator to see
that $\dim W^{+}\cap W^{-}$ has the opposite parity to $n$.  Thus
complementary $W,\hat W$ lie in the same orbit if and only if $n$ is
even.

The real forms of $J^{\pm}(\C^{2n})$ arise from real or quaternionic
structures $j$ on $\C^{2n}$ compatible with $q$: $j^{*}q=\cl{q}$.
This leads to two orbits $J^{\pm}(\R^{n,n})$ of real maximal
isotropic subspaces of $\R^{n,n}$ and one\footnote{Only one of
$J^{\pm}(\C^{4m})$ can contain quaternionic elements since the
intersection of any two such has even dimension.} quaternionic orbit
$J(\HH^{2m})$ when $n=2m$.

The choice of a maximal compact subgroup of $\rSO(n,n)$ amounts to a
decomposition $\R^{n,n}=\R^{n,0}\oplus\R^{0,n}$ into orthogonal
definite spaces.  Now any $W\in J^{\pm}(\R^{n,n})$ is precisely the
graph of an anti-isometry $\R^{n,0}\to\R^{0,n}$ and
$\rSO(n)=\rSO(n,0)$ acts freely and transitively on these by
precomposition.  In this way, both $J^{\pm}(\R^{n,n})$ are identified
as $K$-spaces with the group manifold $\rSO(n)$.  It is an amusing
exercise to see that, in this setting, inverse-stereoprojection
amounts to the Cayley parametrisation of $\rSO(n)$.

A similar analysis can also be made for $\C^{2n}$ equipped with a
complex symplectic form and leads to the complex Lagrangian
Grassmannian $\Lag(\C^{2n})$ with real forms $\Lag(\R^{2n})$ and
$\Lag(\HH^{2m})$, the latter being isomorphic as a $K$-space to the
group manifold $\rSp(m)$.

\subsection{Duality of $R$-spaces}
\label{sec:duality-r-spaces}

Let $M$ be an $R$-space.  Let $M^*$ be the set of parabolic
subalgebras $\fq$ of $\fg$ for which $\fq$ is complementary to some
$\fp\in M$.  In the examples we have just inspected, $M^*$ is a
single conjugacy class which may or may not coincide with $M$.  This
is an entirely general phenomenon and leads to a notion of duality
among $R$-spaces which will be important to us. 

\begin{prop}
$M^*$ comprises a single conjugacy class and so is an $R$-space of
the same height as $M$.
\end{prop}

\begin{proof}
Let $\fq,\fq'\in M^*$ with $\fq$ complementary to $\fp\in M$ and
$\fq'$ complementary to $g\act\fp$.  We must show that $\fq,\fq'$ are
conjugate.  Now $g^{-1}\act\fq'$ is complementary to $\fp$ so by
Lemma~\ref{lemm:simptrans} there is a unique $n\in\exp\fp^\perp$ with
$ng^{-1}\act\fq'=\fq$ and we are done.

$M$ and $M^*$ have the same height since the heights of
complementary parabolic subalgebras coincide. 
\end{proof}

Clearly $(M^*)^*=M$ so we make the following:
\begin{defn}
$M^*$ is the \emph{dual} of $M$. 
\end{defn}

It can happen that $M=M^*$ in which case we say that $M$ is
\emph{self-dual}.  Otherwise, $M$ is said to be \emph{non-self-dual}. 
For example, we have seen that the conformal $n$-sphere is self-dual
but the Grassmannians $G_k(\R^n)$, for $n\neq 2k$, are non-self-dual. 

Even in the non-self-dual case, $M$ and $M^*$ are non-canonically
diffeomorphic: for $\ci$ the Cartan involution corresponding to a
maximal compact $K<G$, the map $\fp\mapsto\ci\fp:M\to M^*$
is a diffeomorphism, in fact an isomorphism of $K$-spaces. 



In \MySec\ref{sec:parab-subalg}, we described how $R$-spaces are
parametrised by certain subsets of nodes of a Dynkin or Satake
diagram. One can detect the duality relation from a certain
automorphism of the diagram. For this, first note that $\fp,\fq\leq\fg$
are complementary if and only if $\fp^\C,\fq^\C\leq\fg^\C$ are, so
that $M$ is self-dual if and only if its complexification
$M^\C=\{g\act\fp^\C:\fp\in M,g\in G^\C\}$ is self-dual (recall $\fp$ is
$G$-conjugate to $\fq$ if and only if $\fp^\C$ is $G^\C$-conjugate to
$\fq^\C$). This reduces the question to the case when $\fg$ is
complex. For $\fp\leq\fg$ parabolic, fix a Borel subalgebra
$\mathfrak{b}\leq\fp$ and a Cartan subalgebra $\fh<\mathfrak{b}$. 
The Borel subalgebra determines a Weyl chamber $C$ and there is a
unique element $w_0$ of the Weyl group such that $w_0 C=-C$
\cite[V,1.6]{Bourbaki1968}. Define $\sigma=-w_0$ so that $\sigma$
preserves $C$ and therefore permutes the corresponding simple roots
giving an automorphism of the Dynkin diagram. If $\fp\in M$
corresponds to a set of simple roots $I$, then the parabolic
subalgebra corresponding to $\sigma I$ lies in $M^*$: it is conjugate
via $w_0$ to the opposite parabolic subalgebra to $\fp$ which
comprises $\fh$ and the root spaces $\fg_{-\alpha}$ for
$\fg_\alpha\subset\fp$ and this last is clearly complementary to
$\fp$. 

Thus our duality relation is implemented on subsets of the Dynkin
diagram by $\sigma$.  For $\fg$ simple, $\sigma$ is readily
determined: it is the identity if and only if $-1$ lies in the Weyl
group (since the Weyl group acts simply transitively on Weyl
chambers) and this is the case unless $\fg$ is one of $\fsl(n)$,
$n>1$; $\fso(2n)$, $n$ odd; or $\mathfrak{e}_6$
\cite[VI,4.5--13]{Bourbaki1968}. In these cases, $\sigma$ is the
usual involution of the Dynkin diagram. 

It is now a simple matter to determine which $R$-spaces are
self-dual: we simply ask that $\sigma$ preserve the corresponding set
of crossed nodes.  The results for symmetric $R$-spaces are tabulated
in Table~\ref{tab:r-space} on page~\pageref{tab:r-space}.

\subsection{The space of complementary pairs}
\label{sec:space-compl-pairs}

Let $M$ be a symmetric $R$-space.  While the Riemannian symmetric
structure of $M$ will not concern us, there is a pseudo-Riemannian
symmetric $G$-space associated to $M$ which will play an important
role in what follows.  This is the \emph{space of complementary
pairs} $Z\subset M\times M^*$:
\begin{equation*}
Z=\{(\fp,\fq)\in M\times M^*:\text{$\fp$ and $\fq$ are complementary}\}. 
\end{equation*}

We have:
\begin{prop}\label{th:3}
$G$ acts transitively on $Z$ so that $Z$ is a pseudo-Riemannian
symmetric $G$-space.  
\end{prop}
\begin{proof}
Let $(\fp,\fq),(\fp',\fq')\in Z$.  There is $h\in G$ with
$h\act\fp=\fp'$ and then $h^{-1}\act\fq'$ is complementary to
$\fp$.  Lemma~\ref{lemm:simptrans} now supplies
$n\in\exp\fp^\perp\leq\Stab(\fp)$ with $n\act\fq=h^{-1}\act\fq'$
and then with $g=hn$ we immediately get $g\act\fp=\fp'$ and
$g\act\fq=\fq'$.  Thus $G$ acts transitively on $Z$. 

The stabiliser of $(\fp,\fq)\in Z$ is clearly
$\Stab(\fp)\cap\Stab(\fq)$ with Lie algebra $\fp\cap\fq$.  This last
is the centraliser of the grading element $\xi_{\fp}^{\fq}$ of the pair
so that $\Stab(\fp)\cap\Stab(\fq)$ is open in the fixed set of the
involution $\tau_{\fp,\fq}=\exp(i\pi\xi)\in\Aut(\fg)$.  Thus $Z$
is a symmetric $G$-space with invariant pseudo-Riemannian metric
induced by the Killing form of $\fg$. 
\end{proof}

Kaneyuki \cite[Theorem~3.1]{Kaneyuki1987} shows that $Z$ is a dense
open subset of $M\times M^*$.  Moreover $Z$ is a \emph{parahermitian
symmetric space} in the sense of Kaneyuki--Kozai \cite{Kaneyuki1985}
who show that essentially all such spaces for semisimple $G$ arise
this way. 

\subsection{Homogeneous geometry and the solder form}
\label{sec:homog-geom-sold}

The geometry of a homogeneous $G$-space is tied to the Lie theory of
$G$ via the \emph{solder form} which is defined as follows.

Let $N$ be a homogeneous $G$-space.  Let $H_x\leq G$ be the
stabiliser of $x\in N$ with Lie algebra $\fh_x$.  Since $G$ acts
transitively, we have a surjection $\fg\to T_xN$ given by
\begin{equation*}
\zeta\mapsto\diffat{t}{t=0}\exp(t\zeta)\act x
\end{equation*} 
with kernel $\fh_x$ and so an isomorphism $\fg/\fh_x\cong T_xN$ whose
inverse is the solder form $\beta^N_x:T_xN\to\fg/\fh_{x}$.  Globally,
the $\fh_x$ are the fibres of a vector subbundle $\fh$ of the trivial
bundle\footnote{We shall consistently denote the trivial
bundle over a manifold $N$ with fibre $V$ by ${\ul V}_{N}$ or even
just $\ul V$ if the context is clear.} ${\ul\fg}_N=N\times\fg$ and
the solder form is then a bundle isomorphism
$\beta^{N}:TN\cong{\ul\fg}_{N}/\fh$.  Dually, the Killing form now
gives an isomorphism $T^{*}N\cong\fh^{\perp}\leq\ul\fg$.

In particular, applying this analysis to an $R$-space $M$, the bundle
of infinitesimal stabilisers is just the tautological bundle of
parabolic subalgebras
\begin{equation*}
\fh_{\fp}=\fp
\end{equation*}
and we have a canonical identification of $T_{\fp}^{*}N$ with
$\fp^{\perp}$, for each $\fp\in M$.  It follows that the fibration of
the symmetric space $Z\subset M\times M^{*}$ over $M$ is
non-canonically isomorphic to the cotangent bundle $T^*M\to M$
\cite[Theorem~4.3]{Kaneyuki1985}: indeed, the fibre $Z_\fp$ at $\fp$
is an affine space modelled on $\fp^\perp$ (see
Lemma~\ref{lemm:simptrans} and the remarks following it) so that any
section of $Z\to M$ allows us to identify $Z_\fp$ with $\fp^\perp$
and so with $T^*_\fp M$.  In particular, a Cartan involution $\ci$ of
$\fg$ with fixed set $K$ gives such a section $\fp\mapsto
(\fp,\ci\fp)$ and $Z\cong T^*M$ as $K$-spaces.  Moreover, this
section gives a totally geodesic embedding of $M$, \emph{qua}
Riemannian symmetric $K$-space, in $Z$.  We shall return to this last
point in \MySec\ref{sec:curv-flats-riem}.

The following lemma provides a helpful characterisation of the solder
form:
\begin{lemm}
\label{th:solder}
Let $s\in\Gamma\fh$ then
\begin{equation}\label{eq:18}
\D s\equiv [\beta^N,s]\mod\fh.
\end{equation}
Moreover, if the fibres of $\fh$ are self-normalising, $\beta^N$ is
the unique ${\ul\fg}_{N}/\fh$-valued $1$-form with this property.
\end{lemm}
\begin{proof}
Let $X\in T_xN$, $s\in\Gamma\fh$ and $\zeta\in\fg$ a representative
of $\beta^N_X\in\fg/\fh_x$.  For any $g\in G$, $\fh_{g\act
x}=\Ad(g)\fh_{x}$ so that $s(\exp(t\zeta)\act
x)=\Ad\exp(t\zeta)\sigma(t)$ for some curve $\sigma$ in $\fh_x$
with $\sigma(0)=s(x)$.  Differentiating this at $t=0$ yields
\begin{equation*}
\D_Xs=[\zeta,s(x)]+\sigma'(0)\equiv [\beta^N_X,s(x)]\mod\fh.
\end{equation*}
The uniqueness assertion is clear.
\end{proof}

If $N$ is reductive\footnote{$R$-spaces are emphatically \emph{not}
reductive but the symmetric spaces $Z$ of complementary pairs are.}
so that each $\fh_{x}$ has an $\Ad H_{x}$-invariant complement
$\fm_{x}$, we have a decomposition into $H$-bundles
\begin{equation*}
\ul\fg=\fh\oplus\fm
\end{equation*}
which allows us to identify $\ul\fg/\fh$ with $\fm$ and so view
$\beta^{N}$ as an $\fm$-valued $1$-form.  With this understood, the
corresponding reduction of $\D$ to an $H$-connection now reads
\begin{equation*}
\D=\cD+\beta^{N}
\end{equation*}
(see, for example, \cite[Chapter 1]{Burstall1990}).

\section{Isothermic maps}
\label{sec:isoth-subm}

We now come to the main object of our discussion: the isothermic maps
to and submanifolds of a symmetric $R$-space and their transformation
theory.

Henceforth, we fix a symmetric $R$-space $M$, a conjugacy
class of parabolic subalgebras of $\fg$, with dual space $M^*$ and
space of complementary pairs $Z\subset M\times M^*$.

\subsection{Definition and $G$-invariance}
\label{sec:defin-zero-curv}

Let $f:\Sigma\to M$ be a map of a manifold $\Sigma$.  We identify $f$
with a subbundle, also called $f$, of parabolic subalgebras of the
trivial bundle $\ul\fg=\Sigma\times\fg$ via
\begin{equation*}
f_x=f(x),
\end{equation*}
for $x\in\Sigma$.  Any such subbundle over connected $\Sigma$ arises
in this way from a map to an $R$-space since the conjugacy classes
are the components of the set of all parabolic subalgebras inside the
Grassmannian of $\fg$. 

\begin{defn}
A map $f:\Sigma\to M$ is \emph{isothermic} if there is a
non-zero closed $1$-form $\eta\in\Omega^1_\Sigma(f^\perp)$. 

If $f$ immerses, $(f,\eta)$ is called an \emph{isothermic submanifold} of
$M$. 
\end{defn}

Thus $\eta$ is a $1$-form taking values in the bundle
$f^\perp$ of nilradicals which is closed when viewed as a
$\fg$-valued form.  Alternatively, the solder isomorphism identifies
$f^{\perp}$ with $f^{-1}T^{*}M$ and we may therefore view $\eta$ as
an $f^{-1}T^{*}M$-valued $1$-form: we shall exploit this in
\MySec\ref{sec:quadratic-form}.

\begin{rem}
For some symmetric $R$-spaces, the $1$-form $\eta$ is uniquely
determined up to (constant) scale by the isothermic map $f$ when $f$
immerses.  However, for other symmetric $R$-spaces, this is far from
the case as we shall see in \MySec\ref{sec:projective-spaces}. 
\end{rem}

We note that the isothermic property is manifestly $G$-invariant: for
isothermic $(f,\eta)$ and $g\in G$, $(g\act f,g\act\eta)$ is also
isothermic. 

\subsection{Stereoprojection and the Christoffel transform}

We saw in \MySec\ref{sec:r-spaces} that a fixed complementary pair
$(\fp_0,\fp_\infty)$ gives us distinguished charts
$\fp_0\in\Omega_{\fp_\infty}\subset M$,
$\fp_\infty\in\Omega_{\fp_0}\subset M^*$ and
stereoprojections $\Omega_{\fp_\infty}\cong\fp_\infty^\perp$,
$\Omega_{\fp_0}\cong\fp_0^\perp$. 

We now show that this data also induces a duality between isothermic
maps.  Indeed, suppose that $f:\Sigma\to M$ has image in
$\Omega_{\fp_\infty}$ and let $F:\Sigma\to\fp_\infty^\perp$ be its
stereoprojection.  Thus
\begin{equation*}
f=\exp(F)\act\fp_0. 
\end{equation*}
If $\eta\in\Omega_\Sigma^1(f^\perp)$, we have a $1$-form
$\omega\in\fp_0^\perp$ such that $\eta=\exp(F)\act\omega$.  Since $F$
takes values in the fixed abelian subalgebra $\fp_\infty^\perp$, we readily
compute the exterior derivative of $\eta$:
\begin{equation*}
\D\eta=\exp(F)\act(\D\omega+[\D F\wedge \omega]). 
\end{equation*}
The two summands on the right lie in $\fp^\perp_0$ and
$\fp_0\cap\fp_\infty$ respectively and so must vanish separately if
and only if $\D\eta=0$.  We therefore conclude:
\begin{prop}\label{th:2}
A map $f:\Sigma\to M$ with stereoprojection
$F:\Sigma\to\fp_\infty^\perp$ with respect to $(\fp_0,\fp_\infty)$ is
isothermic if and only if there is a $1$-form
$\omega\in\Omega^1_\Sigma(\fp_0^\perp)$ such that
\begin{enumerate}
\item $\omega$ is closed;
\item $[\D F\wedge\omega]=0$. 
\end{enumerate}
\end{prop}

In this situation, we can (locally) integrate to find
$F^c:\Sigma\to\fp_0^\perp$ with $\D F^c=\omega$ and thus $[\D F\wedge
\D F^c]=0$.  Everything is now symmetric in $F$ and $F^c$ and we
conclude that $f^c=\exp(F^c)\act\fp_\infty:\Sigma\to M^*$ is
isothermic with $1$-form $\eta^c=\exp(F^c)\act\D F$.  By strict
analogy with the classical situation (see \MySec\ref{sec:exampl-isoth-surf} below), we call
$f^c$ (or $F^c$) the \emph{Christoffel transform} of $f$ (or $F$). 

Of course, if we change our choice of initial complementary pair
$(\fp_0, \fp_\infty)$ then we will obtain a quite different
Christoffel transform.  However, these (or rather their
stereoprojections) are all obtained in a very simple way from a
primitive of $\eta$.  For this, observe that 
\begin{equation*}
\eta=\exp(F)\act\omega=\omega+[F,\omega]+\half[F,[F,\omega]]
\end{equation*}
with the latter three summands taking values in $\fp_0^\perp$,
$\fp_0\cap\fp_\infty$ and $\fp_\infty^\perp$ respectively.  Thus
$\omega=\pi_{\fp_0^\perp}\eta$ for $\pi_{\fp_0^\perp}$ the projection
onto $\fp_0^\perp$ along $\fp_\infty$.  We therefore conclude:
\begin{prop}\label{th:25}
Let $(f,\eta)$ be isothermic and $\Phi:\Sigma\to\fg$ a primitive of
$\eta$: $\D\Phi=\eta$.  Then the Christoffel transform of $(f,\eta)$
with respect to $(\fp_0, \fp_\infty)$ has stereoprojection
$F^c=\pi_{\fp_0^\perp}\Phi$. 
\end{prop}
Thus $\Phi$ is a universal Christoffel transform for $f$.

\subsubsection{Example: isothermic surfaces in $S^n$}
\label{sec:exampl-isoth-surf}

Let us take $M$ to be the conformal $n$-sphere as in
\MySec\ref{sec:example:-conformal-n} and show that, in this case, our
theory recovers the classical notion of an isothermic surface.

A map $f:\Sigma\to M=\pr(\cL)$ is the same as a null line-subbundle
$\Lambda\subset\ul\R^{n+1,1}$ and the bundle of nilradicals is
$f^\perp=\Lambda\wedge\Lambda^\perp$.  Now $f$ is isothermic if there
is a non-zero, closed $\fso(n+1,1)$-valued $1$-form $\eta$ taking
values in $\Lambda\wedge\Lambda^\perp$.  

We stereoproject to compare with the classical definition.  Thus,
choose
$\Lambda_0=\Span{v_0},\Lambda_\infty=\Span{v_\infty}\in\pr(\cL)$ with
$(v_0,v_\infty)=-1$ and, as in
\MySec\ref{sec:example:-conformal-n}, identify both
$\stab(\Lambda_0)$ and $\stab(\Lambda_\infty)$ with
$\R^n=\Span{v_0,v_\infty}^\perp$.  With both $F,F^c$ viewed as maps
$\Sigma\to\R^n$, the bracket $[\D F\wedge\D F^c]$ takes values in
$\fso(n)\oplus\fso(1,1)$ and the vanishing of these two components
amounts to
\begin{align*}
(\D F_X,\D F^c_Y)&=(\D F_Y,\D F^c_X)\\
\D F_X\wedge\D F^c_Y&=\D F_Y\wedge\D F^c_X,
\end{align*}
for vector fields $X,Y$ on $\Sigma$. 
These, in turn, can be conveniently packaged as
\begin{equation}
\label{eq:3}
\D F\wedge_{\Cl_n}\D F^c=0
\end{equation}
where here we use Clifford multiplication to multiply coefficients
so that the left hand side of \eqref{eq:3} is a $2$-form with values
in the Clifford algebra of $\R^n$.  To summarise:
\begin{prop}
$F:\Sigma\to\R^n$ is the stereoprojection of an isothermic map
$\Sigma\to\pr(\cL)$ if and only if there is (locally) a map
$F^c:\Sigma\to\R^n$ such that \eqref{eq:3} holds. 
\end{prop}

According to \cite[\S2.1]{Burstall2004}, this coincides with the
classical formulation of the isothermic property due to Christoffel
\cite{Christoffel1867}, for $n=3$, and Palmer \cite{Palmer1988} for
$n>3$. 

One can show \cite[Lemma~1.6]{Burstall2004} that an isothermic
$F:\Sigma\to\R^n$ has rank at most $2$ so that isothermic
submanifolds are necessarily surfaces.  As we shall see in
\MySec\ref{sec:dimension-bounds}, similar restrictions hold for any
symmetric $R$-space.

\subsection{Zero curvature representation and spectral deformation}
\label{sec:zero-curv-repr}

Isothermic maps comprise an integrable system with a transformation
theory that exactly mirrors the classical theory of isothermic
surfaces in $S^3$.  This is all a consequence of the following simple
observation: for $f:\Sigma\to M$ and
$\eta\in\Omega^1_{\Sigma}(f^{\perp})$, view $\eta$ as a gauge
potential and contemplate the pencil of $G$-connections on $\ul\fg$
given by
\begin{equation}
\label{eq:2}
\Dt=\D+t\eta,
\end{equation}
for $t\in\R$.  Of course, $\nabla^0=\D$. 

We have:
\begin{prop}
\label{th:1}
$(f,\eta)$ is isothermic if and only if $\Dt$ is a flat connection for each $t\in\R$. 
\end{prop}
\begin{proof}
We compute:
\begin{equation*}
R^{\Dt}=R^\D+t\D\eta+\half t^2[\eta\wedge \eta]=t\D\eta
\end{equation*}
since $\D$ is flat and $\eta$ takes values in the bundle of
abelian\footnote{It is here that $M$ being a \emph{symmetric}
$R$-space comes to the fore.} subalgebras $f^\perp$.  Thus $\D\eta=0$
if and only if $R^{\Dt}=0$ for all $t\in\R$. 
\end{proof}

As a first application, we show that isothermic maps possess a
\emph{spectral deformation}: that is, they come locally in
$1$-parameter families.  There are two ways to see this corresponding
to the active and passive viewpoints on gauge transformations but the
key observation is that 
\begin{equation*}
\D^{\Dt}\eta=\D\eta+t[\eta\wedge\eta]=0. 
\end{equation*}
Thus $(f,\eta)$ remains isothermic when the flat connection $\D$ on
$\ul\fg$ is replaced by the flat connection $\Dt$. 

Equivalently, we can trivialise $\Dt$:  locally, we can find
gauge transformations $\Phi_t$, unique
up to left multiplication by a constant element of $G$, such that
$\Phi_t\cdot\Dt=\D$ and then
$(\Phi_t f,\Phi_t\eta)$ is isothermic in the usual sense.  Here, and
below, $\Phi_t\cdot\Dt=\Phi_t\circ\Dt\circ\Phi_t^{-1}$ is the usual
left action of gauge transformations on connections.

We write $\ttrans_t f$ for $\Phi_t f$ and, following Bianchi
\cite{Bianchi1905}, say that $\ttrans_t f$ is a
\emph{$T$-transform}\footnote{In the case of isothermic surfaces in
$S^3$, it will follow from the results of \MySec\ref{sec:curved-flats-z}
that our notion of $T$-transform coincides with the classical one due
to Calapso \cite{Calapso1903} and Bianchi \cite{Bianchi1905} since
both are given by the curved flat spectral deformation of Darboux
pairs \cite{Burstall1997a}.} of
$f$. 

The gauge transformation $\Phi_t$ does more than intertwine $\Dt$ and
$\D$: we clearly have $\Phi_t\cdot(\Dt+s\eta)=\D+s\Phi_t\eta$, for
all $s\in\R$.  Now, if $\Phi_s^t$ implements the transform
$\ttrans_s$ of $\ttrans_t f$ so that
$\Phi_s^t\cdot(\D+s\Phi_t\eta)=\D$, we have
$(\Phi^t_s\Phi_t)\cdot\nabla^{s+t}=\D$ so that
$\Phi^t_s\Phi_t=\Phi_{s+t}$.  We therefore obtain an identity on
$T$-transforms due to Hertrich-Jeromin--Musso--Nicolodi
\cite{Hertrich-Jeromin2001} for the case where $M$ is the conformal
$3$-sphere:
\begin{equation}\label{eq:9}
\ttrans_{s+t}=\ttrans_s\circ\ttrans_t. 
\end{equation}

\subsection{Darboux transforms}\label{sec:darboux}

Flat connections have many parallel sections locally and, in our
situation, we can use the parallel sections of $\Dt$ to construct new
isothermic surfaces in the dual symmetric $R$-space. 

\begin{defn}
Let $(f,\eta)$ be isothermic and $m\in\R^\times$.  A map $\hat
f:\Sigma\to M^*$ into the dual $R$-space is a \emph{Darboux transform
of $f$ with parameter $m$} if
\begin{enumerate}
\item $\hat f \subset \ul\fg$ is $\nabla^m$-parallel;
\item $f$ and $\hat f$ are pointwise complementary parabolic subalgebras. 
\end{enumerate}
In this case, we write $\hat f=\darb_m f$. 
\end{defn}

\begin{rem}\item[]
\begin{enumerate}
\item The condition that $\hat f$ be $\nabla^m$-parallel means that
any section $\zeta$ of $\hat f$ has
$\nabla^m\zeta\in\Omega^1_\Sigma(\hat f)$.  Since $\nabla^m$ is
flat, such $\hat f$ are determined by their value at a fixed point
$p_0\in \Sigma$ and any complement to $f(p_0)$ extends locally to a
parallel $\hat f$ (which may eventually fail to be pointwise
complementary to $f$). 
\item That $(f,\hat f)$ be pointwise complementary amounts to
demanding that $(f,\hat f)$ takes values in the symmetric space $Z$
of complementary pairs.  As we shall see in \MySec\ref{sec:curved},
$\hat f$ is a Darboux transform if and only if $(f,\hat f)$ is a
\emph{curved flat} in the sense of Ferus--Pedit \cite{Ferus1996}. 
\end{enumerate}
\end{rem}

We are going to show that $\hat f$ is isothermic and that the
relationship between $f$ and $\hat f$ is reciprocal: $f=\darb_m \hat
f$.  In fact, we shall show more and exhibit an explicit gauge
transformation between the two pencils of flat connections.  All of
this will take a little preparation. 

So let $(f,\eta)$ be isothermic, $\hat f=\darb_m f$ and contemplate
the bundle decomposition
\begin{equation}
\label{eq:5}
\ul\fg=f^\perp\oplus (f\cap\hat f)\oplus \hat f^\perp. 
\end{equation}
This is the eigenbundle decomposition of $\ad\xi$ where
$\xi(p)=\xi_{f(p)}^{\hat f(p)}$ is the grading element of the
complementary pair at $p\in\Sigma$.  There is a corresponding
decomposition of the flat connection $\D$: since $\xi$ takes values
in a single conjugacy class, $\D\xi$ takes values in
$\im\ad\xi=f^\perp+\hat f^\perp$ on which $\ad\xi$ is invertible and
we may write $\D\xi=[\xi,\beta+\hat\beta]$ where
$\beta\in\Omega^1_\Sigma(\hat f^\perp)$ and $\hat
\beta\in\Omega^1_\Sigma(f^\perp)$.  Now define a new
$G$-connection $\cD$ by
\begin{equation}
\label{eq:4}
\D=\cD-\beta-\hat\beta
\end{equation}
and note that $\cD\xi=0$ so that each summand in \eqref{eq:5} is
$\cD$-parallel. 

We have that $\hat f$ is $\nabla^m$-parallel.  Since $\hat f$ is also
$\cD$-parallel and preserved by $\ad\beta$, we conclude that it is also
preserved by $\ad(m\eta-\hat\beta)$.  Since $\hat f$ is
self-normalising, this gives that $m\eta-\hat\beta$ takes values in
$\hat f\cap f^\perp=\{0\}$ so that $\hat\beta=m\eta$. 

The symmetry of the situation now suggests that if we define
$\hat\eta=\tfrac{1}{m}\beta\in\Omega^1_\Sigma(\hat f^\perp)$
then $(\hat f,\hat\eta)$ should be isothermic.  This will be the case
if and only if the connections $\hat\nabla^t=\D+t\hat\eta$ are flat
for all $t\in\R$ and we will prove this by writing down an explicit
gauge transformation intertwining $\Dt$ and $\hat\nabla^t$. 

For this, we introduce a class of elements of $\Aut(\fg)$ that will
appear frequently in what follows.  For $(\fp,\fq)\in M\times M^*$
complementary and $s\in\R^\times$, define $\Gamma_\fp^\fq(s)$ by
\begin{equation}
\label{eq:6}
\Gamma_\fp^\fq(s)=
\begin{cases}
s& \text{on $\fq^\perp$}\\
1& \text{on $\fp\cap\fq$}\\
s^{-1}& \text{on $\fp^\perp$}\end{cases}
\end{equation}
We note that $s\mapsto\Gamma_\fp^\fq(s)$ is a homomorphism
$\R^\times\to\Aut(\fg)$, that $\Gamma_\fp^\fq(s)=\exp((\ln
s)\act\xi_{\fp}^{\fq})\in G$, for $s>0$, and that $\Gamma(-1)$ is the
involution at $(\fp,\fq)$ defining the symmetric space structure of
$Z$ in Proposition~\ref{th:3}. 

\begin{rem}
These $\Gamma_\fp^\fq$ are the analogues, for real semisimple $G$,
of the \emph{simple factors} considered by Terng--Uhlenbeck
\cite{Terng2000}.  We shall return to this point in
\MySec\ref{sec:curved-flats}. 
\end{rem}

Now consider the action of the gauge transformation $\Gamma_f^{\hat
f}(s)$ on the connection $\Dt$:
\begin{equation}\label{eq:7}
\begin{split}
\Gamma_f^{\hat f}(s)\cdot\Dt&=\Gamma_f^{\hat f}(s)\cdot
(\cD-\beta-\tfrac{m-t}{m}\hat\beta)\\
&=\cD - s\beta-s^{-1}\tfrac{m-t}{m}\hat\beta,
\end{split}
\end{equation}
since $\Gamma_f^{\hat f}(s)$ is $\cD$-parallel, having constant
eigenvalues and $\cD$-parallel eigenspaces, and both $\beta$,
$\hat\beta$ take values in eigenspaces of $\Gamma_f^{\hat f}$. 

In particular, 
\begin{equation*}
\Gamma_f^{\hat f}(1-\tfrac{t}{m})\cdot\Dt=\cD
-\tfrac{m-t}{m}\beta-\hat\beta=\hat\nabla^t 
\end{equation*}
whence each $\hat\nabla^t$ is flat (for $t\neq m$ and then for all
$t$ by continuity) since $\Dt$ is flat.  Moreover,
$\hat\nabla^m=\cD-\hat\beta$ so that $f$ is clearly
$\hat\nabla^m$-parallel, that is $f=\darb_m \hat f$.  To summarise:
\begin{thm}
\label{th:4}
Let $(f,\eta)$ be isothermic and $\hat f=\darb_m f$ a Darboux
transform of $f$ with $\hat\eta$ defined as above.  Then:
\begin{enumerate}
\item $(\hat f,\hat\eta)$ is isothermic;
\item $\Gamma_f^{\hat f}(1-\frac{t}{m})\cdot(\D+t\eta)= \D+t\hat\eta$ for all $t\neq m$;
\item $f=\darb_m \hat f$. 
\end{enumerate}
\end{thm}
There is a converse to Theorem~\ref{th:4} which detects when two
isothermic surfaces are Darboux transforms of each other:
\begin{prop}
\label{th:24}
Let $(f,\eta)$, $(\hat f,\hat\eta)$ be pointwise complementary isothermic
surfaces such that $\Gamma_f^{\hat
f}(1-\frac{t}{m})\cdot(\D+t\eta)=\D+t\hat\eta$, for all $t\neq m$.

Then $\hat f=\darb_mf$ (whence $f=\darb_m\hat f$).
\end{prop}
\begin{proof}
As above, write $\D=\cD-\beta-\hat\beta$ so that we have
\begin{equation*}
\cD-\tfrac{m-t}{m}\beta-\tfrac{m}{m-t}(\hat\beta-t\eta)=\cD-\beta-\hat\beta+t\hat\eta
\end{equation*}
and compare components in $f^{\perp},\hat f^{\perp}$ to get
\begin{equation*}
m\beta=\hat\eta,\qquad m\hat\beta=\eta.
\end{equation*}
In particular, $\D+m\eta=\cD-\beta$ so that $\hat f$ is
$(\D+m\eta)$-parallel and thus a Darboux transform of $f$ with
parameter $m$.
\end{proof}

Part~(2) of Theorem~\ref{th:4} enables us to prove a permutability
theorem relating Darboux and $T$-transforms: indeed, with $\hat
f=\darb_s f$ and $\hat\Dt=\D+t\hat\eta$, the $T$-transform $\ttrans_f
\hat f$ is given by $\hat\Phi_t\hat f$ for $\hat\Phi_t$ a gauge
transformation with $\hat\Phi_t\cdot\hat\Dt=\D$.  On the other hand,
the $T$-transform of $f$ is implemented by $\Phi_t$ with
$\Phi_t\cdot\Dt=\D$ so that we conclude that
\begin{equation*}
\hat\Phi_t\Gamma_f^{\hat f}(1-\tfrac{t}{s})=\Phi_t, 
\end{equation*}
up to left multiplication by constants in $G$.  Apply this to $\hat
f$ to conclude that $\hat\Phi_t\hat f=\Phi_t \hat f$.  Now, $\hat f$
is $\nabla^s=\Dt+(s-t)\eta$-parallel, whence $\Phi_t \hat f$ is
$\D+(s-t)\Phi_t\eta$-parallel.  Thus $\Phi_t \hat f$ is a Darboux
transform of $\Phi_t f$ with parameter $s-t$.  To summarise, we have
proved a result which can be found in
\cite[Theorem~5.6.15]{Hertrich-Jeromin2003} for the classical case:
\begin{thm}
\label{th:5}
For $t,s\in\R$ with $0\neq s\neq t$, $\ttrans_t\darb_s=\darb_{s-t}\ttrans_t$. 
\end{thm}

There is a similar result, due to Bianchi \cite[\S3]{Bianchi1905} in
the classical setting, that uses the $T$-transform to intertwine
Christoffel and Darboux transformations.  Again, the key for us is to
find the gauge transformation relating the pencils of flat
connections corresponding to an isothermic map and its Christoffel
transform.  For this, let $(f,\eta)$ be
isothermic and fix a complementary pair $(\fp_0,\fp_\infty)\in
M\times M^*$.  Then we have maps $F:\Sigma\to\fp_\infty^\perp,
F^c:\Sigma\to\fp_0^\perp$ such that $f=\exp(F)\act\fp_0$,
$\eta=\exp(F)\act\D F^c$ while the corresponding Christoffel transform
is $(f^c,\eta^c)$ with $f^c=\exp(F^c)\act\fp_\infty$ and
$\eta^c=\exp(F^c)\act\D F$.  Since $F$ takes values in a fixed abelian
subalgebra, we have
\begin{equation*}
\D=\exp(F)\cdot(\D+\D F)
\end{equation*}
so that
\begin{equation*}
\D+t\eta=\exp(F)\cdot(\D+\D F + t\D F^c)
\end{equation*}
and, similarly,
\begin{equation*}
\D+t\eta^c=\exp(F^c)\cdot(\D+t\D F + \D F^c). 
\end{equation*}
Thus define
$\Gamma^c(t)=\exp(F^c)\Gamma_{\fp_0}^{\fp_\infty}(t)\exp(-F)$ and
conclude:
\begin{lemm}\label{th:6}
With $f,f^c$ a Christoffel pair of isothermic surfaces and
$\Gamma^c(t)$ defined as above,
$\Gamma^c(t)\cdot(\D+t\eta)=\D+t\eta^c$, for all
$t\neq 0$. 
\end{lemm}
Note that $\Gamma^c(t)f\equiv\fp_0$ while
$\Gamma^c(t)\fp_\infty\equiv f^c$. 

We can now argue as for Theorem~\ref{th:5} above: the $T$-transform
of $f^c$ is $\Phi^c_t f^c$ where $\Phi^c_t\cdot(\D+t\eta^c)=\D$. 
Thus Lemma~\ref{th:6} tells us that
\begin{equation*}
\Phi^c_t\Gamma^c(t)=\Phi_t
\end{equation*}
so that $\Phi^c_t f^c=\Phi_t\Gamma^c(t)^{-1}f^c=\Phi_t \fp_\infty$. 
Now the constant section $\fp_\infty$ is $\D=\Dt-t\eta$-parallel so
that $\Phi_t \fp_\infty$ is $\D-t\Phi_t\eta$-parallel, that is, $\Phi_t
\fp_\infty$ is a Darboux transform of $\Phi_t f$ with parameter
$-t$.  To summarise:
\begin{thm}
\label{th:7}
For $t\in\R^\times$, $\ttrans_t f^c=\darb_{-t}\ttrans_t f$. 
\end{thm}

\subsection{Christoffel transform as a blow-up of Darboux transforms}
\label{sec:christ-transf-as}

If we let $t\to 0$ in Theorem~\ref{th:7}, it appears that the Christoffel transform
$f^c$ is some kind of limit of Darboux transforms.  This is indeed
the case as we now show.

Fix $(f,\eta)$ isothermic, $f:\Sigma\to M$, and contemplate the gauge
transformations $\Phi_s:\Sigma\to G$ that implement the
$T$-transforms of $f$.  Thus $\Phi_s\cdot\nabla^s=\D$ and each
$\Phi_s$ is determined uniquely up to left multiplication by a
constant in $\Aut(\fg)$.  In particular, we may assume that
$\Phi_{s}$ depends smoothly on $s$ and that $\Phi_0=1$.  Now define
$\dot\Phi:\Sigma\to\fg$ by
\begin{equation*}
\dot\Phi=\left.\frac{\partial\Phi_s}{\partial s}\right|_{s=0}.
\end{equation*}
Then:
\begin{prop}
$\D\dot\Phi=\eta$.
\end{prop}
\begin{proof}
The defining property of the $\Phi_{s}$ amounts to
\begin{equation*}
\Phi_s^{-1}\D\Phi_{s}=s\eta
\end{equation*}
and we differentiate this with respect to $s$ at $s=0$ to draw the
conclusion.
\end{proof}
Thus $\dot\Phi$ is a primitive for $\eta$ and so, by
Proposition~\ref{th:25}, is a universal Christoffel transform for
$f$.  That is, for $(\fp_0,\fp_{\infty})\in M\times M^{*}$ a
complementary pair with $\fp_{\infty}$ pointwise complementary to $f$,
$f^c:=\exp(\pi_{\fp_0^{\perp}}\dot\Phi)\act\fp_{\infty}:\Sigma\to
M^{*}$ is a Christoffel transform of $(f,\eta)$ with respect to
$(\fp_0,\fp_{\infty})$.

We use this to give two alternative characterisations of $f^c$.
First we see that $f^c$ is a blow-up of Darboux transforms of $f$.
Indeed let $\hat{f}_s:\Sigma\to M^{*}$ be a smooth variation of maps
through $\hat{f}_0\equiv\fp_{\infty}$ with each $\hat{f}_s$ $\nabla^s$-parallel
and pointwise complementary to $f$.  Thus, for $s\neq 0$, $\hat{f}_s$ is a
Darboux transform of $(f,\eta)$ with parameter $s$: $\hat
f_s=\darb_sf$.  Then $\partial/\partial s|_{s=0}\hat f_{s}:\Sigma\to
T_{\fp_{\infty}}M^{*}$. Now the derivative at $1$ of $g\mapsto
g\act\fp_{\infty}$ induces the soldering isomorphism
$\fg/\fp_{\infty}\cong T_{\fp_{\infty}}M^{*}$ and thus we get an
identification $\fp_{0}^{\perp}\cong T_{\fp_{\infty}}M^{*}$.  With this
in mind, observe that we may choose $\Phi_{s}$ so
that $\hat f_s=\Phi_s^{-1}\act\fp_{\infty}$ whence, since
$\partial/\partial s|_{s=0}\Phi_s^{-1}=-\dot\Phi$, we conclude:
\begin{cor}
\label{th:26}
Under the identification $T_{\fp_{\infty}}M^{*}\cong\fp_{0}^{\perp}$,
$\partial/\partial s|_{s=0}\hat{f}_{s}=-\pi_{\fp_{\infty}}\dot\Phi$ so that
\begin{equation*}
f^c=\exp(-\partial/\partial s|_{s=0}\hat{f}_{s})\act\fp_{\infty}.
\end{equation*}
\end{cor}

This leads us to another realisation of both $f^c$ and its
accompanying $1$-form $\eta^{c}$, this time as a limit of conjugates
of $(\hat{f}_s,\hat\eta_s)$ which will be useful below.
\begin{thm}
\label{th:27}
Let $(\hat{f}_s,\hat\eta_s)$ be a smooth family of Darboux transforms
of $(f,\eta)$ with parameter $s$ such that $\lim_{s\to
0}\hat{f}_s=\fp_{\infty}$.  Then a Christoffel transform
$(f^c,\eta^c)$ of $(f,\eta)$ with respect to $(\fp_0,\fp_{\infty})$
is given by
\begin{align*}
f^{c}&=\lim_{s\to0}\Gamma_{\fp_0}^{\fp_{\infty}}(-s)\hat{f}_{s},\\
\eta^{c}&=\lim_{s\to0}\Gamma_{\fp_0}^{\fp_{\infty}}(-s)\hat{\eta}_{s}.
\end{align*}
In particular, $\D+t\eta^c=\lim_{s\to
0}\Gamma_{\fp_0}^{\fp_{\infty}}(-s)\cdot(\D+t\hat\eta_s)$, for all $t\in\R$.
\end{thm}
\begin{proof}
Let $\hat F_s:\Sigma\to\fp_0^{\perp}$ be the stereoprojection of
$\hat f_s$ from $\fp_{0}$ so that $\hat{f}_s=\exp(\hat
F_{s})\act\fp_{\infty}$.  The derivative of stereoprojection at
$\fp_{\infty}$ is precisely our identification
$T_{\fp_{\infty}}M^{*}\cong\fp_0^{\perp}$ under which
$\partial/\partial s|_{s=0}\hat{f}_{s}=\partial/\partial
s|_{s=0}\hat{F}_{s}$ so that Corollary~\ref{th:26} reads
\begin{equation*}
f^c=\exp(-\partial/\partial s|_{s=0}\hat{F}_{s})\act\fp_{\infty}.
\end{equation*}
However, $\lim_{s\to 0}\hat F_{s}=0$ so that
\begin{equation*}
f^c=\lim_{s\to 0}\exp(-\hat F_s/s)\act\fp_{\infty}=
\lim_{s\to
0}\Gamma_{\fp_0}^{\fp_{\infty}}(-s)\exp(\hat F_s)\act\fp_{\infty}=
\lim_{s\to0}\Gamma_{\fp_0}^{\fp_{\infty}}(-s)\hat f_{s}
\end{equation*}
as required.

As for the $1$-forms, recall that $\hat\eta_{s}=\beta_{s}/s$ where
$\beta_s\in\Omega_{\Sigma}^1(\hat f_{s}^{\perp})$ and
$\hat\beta_s\in\Omega_{\Sigma}^1(f^{\perp})$ are given by 
$\D\xi_f^{\hat f_s}=[\xi_f^{\hat f_s},\beta_s+\hat\beta_s]$.
Meanwhile, $\eta^c=\exp(F^{c})\act\D F$ where
$F:\Sigma\to\fp_{\infty}^{\perp}$ is the stereoprojection of $F$ and,
as we have just seen, $F^{c}=-\partial/\partial s|_{s=0}\hat F_{s}$
is the stereoprojection of $f^{c}$.
Write $\hat\eta_s=\exp(\hat F_{s})\act b_{s}/s$ with
$b_s\in\Omega_{\Sigma}^1(\fp_{\infty}^{\perp})$.  Then
\begin{equation*}
\lim_{s\to
0}\Gamma_{\fp_0}^{\fp_{\infty}}(-s)\hat\eta_s=
\lim_{s\to
0}\Gamma_{\fp_0}^{\fp_{\infty}}(-s)\exp(\hat F_{s})\act b_{s}/s=
-\lim_{s\to
0}\exp(-\hat F_{s}/s)\act b_{s}=-\exp(-\partial/\partial s|_{s=0}\hat F_{s})\act b_{0}.
\end{equation*}
Now $b_{0}=\beta_{0}$ and it remains to identify this with $-\D F$.  However,
$\xi_{f}^{\fp_{\infty}}=\exp(F)\act\xi_{\fp_0}^{\fp_{\infty}}$ and we
differentiate to get $\D\xi_{f}^{\fp_{\infty}}=[\D
F,\xi_{f}^{\fp_{\infty}}]$ yielding $\beta_{0}=-\D F$ as required
(and $\hat\beta_0=0$).
\end{proof}

\section{Bianchi permutability in self-dual spaces}
\label{sec:bianchi-perm-self}

Let us recall Bianchi's celebrated Permutability Theorem
\cite[\S5--\S9]{Bianchi1905a} for Darboux
transforms of isothermic surfaces in $\R^3$: given such an isothermic
surface $f$ and two distinct Darboux transforms $f_1=\darb_{m_1}f$,
$f_2=\darb_{m_2}f$ with $m_1\neq m_2$, then there is a fourth
isothermic surface $\hat f$ which is a simultaneous Darboux transform
of $f_1$ and $f_2$: $\hat f=\darb_{m_2}f_1=\darb_{m_1}f_2$.  Moreover,
$\hat f$ can be constructed algebraically from the first three
surfaces: in fact, according to Demoulin \cite[p.~157]{Demoulin1910},
corresponding points on the four surfaces are concircular with
constant cross-ratio $m_2/m_1$.

We now show that these results still hold \emph{in every detail} for
isothermic maps into any \emph{self-dual} symmetric $R$-space.  For
this, we begin by describing a distinguished family of circles in
such a space.
  
\subsection{Circles in self-dual symmetric $R$-spaces}
\label{sec:circles-self-dual}

Three distinct points in $S^n$ determine a unique circle on which
they lie.  This fact generalises to self-dual symmetric $R$-spaces
$M$: three pairwise complementary points in $M$ determine a
submanifold of $M$, conformally diffeomorphic to a circle, on which
they lie.

Let $\fp_0,\fp_\infty\in M$ be complementary points in a self-dual
symmetric $R$-space.  The set of parabolic subalgebras in $M$ that
are complementary to both $\fp_0$ and $\fp_\infty$ is the dense open
set $\Omega_{\fp_0}\cap\Omega_{\fp_\infty}$.  We begin with a simple
criterion for membership of this set.

\begin{lemm}
\label{th:8}
Let $x_\infty\in\fp_\infty^\perp$.  Then
$\exp(x_\infty)\act\fp_0\in\Omega_{\fp_0}\cap\Omega_{\fp_\infty}$ if
and only if $\ker(\ad x_\infty)^2=\fp_\infty$.
\end{lemm}
\begin{proof}
$\exp(x_\infty)\act\fp_0\in\Omega_{\fp_0}\cap\Omega_{\fp_\infty}$ if
and only if $\exp(x_\infty)\act\fp_0\in\Omega_{\fp_0}$, that is,
$\exp(x_\infty)\act\fp_0$ and $\fp_0$ are complementary. This
means that $\exp(x_\infty)\act\fp_0^\perp\cap\fp_0=\{0\}$.  However,
for $x\in\fp_0^\perp$, we have
\begin{equation*}
\exp(x_\infty)\act x=x+[x_\infty, x]+\half(\ad x_\infty)^2x
\end{equation*}
with the first two summands in $\fp_0$ and the last in
$\fp_\infty^\perp$.  Now $\fg=\fp_0\oplus\fp_\infty^\perp$ so
$\exp(x_\infty)\act x\in\exp(x_\infty)\act\fp_0^\perp\cap\fp_0$ if and
only if $(\ad x_\infty)^2x=0$.  Thus
$\exp(x_{\infty})\act\fp_0\in\Omega_{\fp_0}\cap\Omega_{\fp_{\infty}}$
if and only if $(\ad x_{\infty})^2$ injects on $\fp_0^{\perp}$.
Since $\fg=\fp_0^{\perp}\oplus\fp_{\infty}$ and $\fp_\infty\subset\ker(\ad
x_\infty)^2$, for any $x_\infty\in\fp_\infty^\perp$, the lemma follows.
\end{proof}
Note that the condition on $x_\infty$ is independent of the choice
of complementary $\fp_0$. 

Suppose now that we have three mutually complementary points
$\fp_0,\fp_1,\fp_\infty\in M$ and write $\fp_1=\exp(x_\infty)\act\fp_0$
for a unique $x_\infty\in\fp_\infty^\perp$.  For $t\in\R^\times$,
$(\ad tx_\infty)^2=t^2(\ad x_\infty)^2$ and so has kernel
$\fp_\infty$ also whence, by Lemma~\ref{th:8},
$\fp_t:=\exp(tx_\infty)\act\fp_0$ is again complementary to $\fp_0$
and $\fp_\infty$.

\begin{defn}
The \emph{circle through $\fp_0,\fp_1,\fp_\infty$} is the subset
$C\subset M$ given by
\begin{equation*}
C=\set{\fp_t\colon t\in\R\cup\set{\infty}}.
\end{equation*}
\end{defn}
Note that, for $t\neq s$, $\fp_t,\fp_s$ are the images under $\exp
sx_\infty$ of the complementary pair $(\fp_{t-s},\fp_0)$ and so are
also complementary.

We now show that our construction is independent of choices, that is,
any three points of $C$ determine the same circle.  For this we give
an alternative approach to $C$ which also shows that the projective
structure on $C$ given by the coordinate $t$ is also independent of
choices.

For $i\neq j\in\set{0,1,\infty}$, let $\xi^i_j$ be the grading
element of the pair $(\fp_j,\fp_i)$.  Thus
\begin{equation*}
\ad\xi^i_j=
\begin{cases}
1&\text{on $\fp_i^\perp$}\\
0&\text{on $\fp_i\cap\fp_j$}\\
-1&\text{on $\fp_j^\perp$}
\end{cases}
\end{equation*}
and $\xi^j_i=-\xi^i_j$. 

\begin{prop}
\label{th:9}
The span $\fs=\Span{\xi_0^1,\xi_1^\infty,\xi_0^\infty}\subset\fg$ is
a subalgebra of $\fg$ isomorphic to $\fsl(2,\R)$.  Indeed, with
$\fp_1=\exp(x_\infty)\act\fp_0=\exp(x_0)\act\fp_\infty$, for
$x_i\in\fp_i^\perp$, we have $x_0,x_\infty\in\fs$ and
\begin{align*}
[\xi_0^\infty,x_\infty]&=x_\infty &
[x_\infty,x_0]&=2\xi_0^\infty&
[\xi_0^\infty,x_0]&=-x_0.
\end{align*}
\end{prop}
\begin{proof}
For distinct $i,j,k\in\set{0,1,\infty}$, there is, by
Lemma~\ref{lemm:simptrans}, $x_{jk}\in\fp_i^\perp$ such that
$\exp(x_{jk})\act\xi^j_i=\xi_i^k$.  That is,
\begin{equation}\label{eq:8}
\xi_i^k=\xi_i^j+[x_{jk},\xi_i^j]=\xi_i^j+x_{jk}.
\end{equation}
Thus $\xi_i^k-\xi_i^j=x_{jk}\in\fp_i^\perp$ so that
\begin{equation}
\label{eq:10}
[\xi_i^j,\xi_i^k]=[\xi_i^j,\xi_i^k-\xi_i^j]=\xi_i^j-\xi_i^k
\end{equation}
which shows that $\fs$ is a subalgebra.

Moreover, we have $\exp(x_\infty)\act\xi^0_\infty=\xi_\infty^1$ so
\eqref{eq:8} gives $x_\infty=\xi_\infty^1-\xi_\infty^0\in\fs$ and,
similarly, $x_0=\xi_0^1-\xi_0^\infty\in\fs$.  Now use \eqref{eq:10}
and $\xi_i^j=-\xi_j^i$ to compute:
\begin{align*}
[x_\infty,x_0]&=[\xi_\infty^1-\xi_\infty^0,\xi_0^1-\xi_0^\infty]\\
&=[\xi_1^\infty,\xi_1^0]+[\xi_\infty^1,\xi_\infty^0]+[\xi_0^\infty,\xi_0^1]\\
&=2\xi_0^\infty.
\end{align*}
\end{proof}

The Killing form $(\,,\,)_\fs$ of $\fs$ has signature $(2,1)$
(indeed, $x_0, x_\infty$ are null and
$(\xi_0^\infty,\xi_0^\infty)_\fs=2$).  The light-cone
$\cL_\fs\subset\fs$ for this inner product is precisely the set of
nilpotent elements of $\fs$.  The projective light-cone
$\pr(\cL_{\fs})$ is then, on the one hand, a conformal diffeomorph
of $S^1$ and, on the other, the set of nilradicals of the single
conjugacy class of parabolic subalgebras of $\fs$.

Let $S\leq G$ be the analytic subgroup of $G$ with Lie algebra $\fs$.
Then $S\cong\rPSL(2,\R)$
and acts transitively on $\pr(\cL_\fs)$.  Thus any
$\Span{x}\in\pr(\cL_\fs)$ is of the form $h\act\Span{x_\infty}$, for
some $h\in S$, so that $\ker(\ad x)^2=h\act\ker(\ad
x_\infty)^2=h\act\fp_\infty\in M$.  We have therefore defined an
$S$-equivariant map $\Psi:\pr(\cL)\to M$ by
\begin{equation*}
\Psi(\Span{x})=\ker(\ad x)^2
\end{equation*}
which injects since $\fs\cap\Psi(\Span{x})^\perp=\Span{x}$.  It is
easy to see that the circle $C$ constructed earlier coincides with
the image of $\Psi$: indeed, since $\fp_0=\Psi(\Span{x_0})$, the
$S$-equivariance of $\Psi$ gives
\begin{equation}\label{eq:11}
\fp_t=\exp(tx_\infty)\act\fp_0=\exp(tx_\infty)\act\Psi(\Span{x_0})=
\Psi(\exp(tx_\infty)\act\Span{x_0}),
\end{equation}
for $t\in\R$.  Moreover, since $S$ acts (simply) transitively on
triples of distinct points of $\pr(\cL_\fs)$, we see that any three
distinct (and so complementary) $\fp'_0,\fp'_1,\fp'_\infty\in C$
define the same $\fs$ (the grading elements $\xi'{}_i^j$ lie in $\fs$
and therefore span it) and thus the same $S$ and $C$.

We have now equipped our circle $C$ with a conformal structure, or,
what is the same thing in dimension one, a projective structure.
Indeed, fix a double cover $\rSL(2,\R)\to S$ and then we have an
equivariant isomorphism $\R\pr^1\cong \pr(\cL_\fs)$ given by
$\ell\mapsto\stab(\ell)^\perp$.  From \eqref{eq:11}, we see that the
coordinate $\fp_t\mapsto t$ on $C$ is the pull-back by $\Psi^{-1}$ of the
coordinate on $S^1$ given by stereoprojection.  According to
\MySec\ref{sec:grassmannians}, this last is an affine coordinate on
$\R\pr^1$.

Here are some consequences of this circle of ideas.  First, four
distinct points on $C$ have, via the identification with $\R\pr^1$,
an $S$-invariant cross-ratio and, in particular, the cross-ratio of
$\fp_1,\fp_\infty,\fp_t,\fp_0$ is exactly $t$:
\begin{equation}
\label{eq:12}
\cross(\fp_1,\fp_\infty,\fp_t,\fp_0)=t.
\end{equation}
 
For the second, let us first relate the parametrisation
$t\mapsto\fp_t$ to the gauge transformations $\Gamma_{\fp}^\fq$ of
\MySec\ref{sec:darboux}: for $t\in\R^\times$,
\begin{align*}
\Gamma_{\fp_0}^{\fp_\infty}(t)\fp_1&=\Gamma_{\fp_0}^{\fp_\infty}(t)\exp(x_\infty)\act\fp_0\\
&=\exp\bigl(\Gamma_{\fp_0}^{\fp_\infty}(t)x_\infty\bigr)\bigl(\Gamma_{\fp_0}^{\fp_\infty}(t)\fp_0\bigr)\\
&=\exp(tx_\infty)\act\fp_0=\fp_t,
\end{align*}
since $\Gamma_{\fp_0}^{\fp_\infty}(t)$ preserves $\fp_0$ and has
$\fp_\infty^\perp$ as eigenspace with eigenvalue $t$.  Thus we may
set $\Gamma_{\fp_0}^{\fp_\infty}(0)\fp_1=\fp_0$,
$\Gamma_{\fp_0}^{\fp_\infty}(\infty)\fp_1=\fp_\infty$ and see that
$t\mapsto\Gamma_{\fp_0}^{\fp_\infty}(t)\fp_1$ is the parametrisation
of $C$ inverting the affine coordinate $\fp_t\mapsto t$.  Now any two
affine coordinates on $\R\pr^1$ are related by a unique linear
fractional transformation so the same is true of the corresponding
parametrisations.  We summarise the situation in the following
proposition which will enable us to avoid several tedious
computations below.
\begin{prop}\label{th:10}
Let $(\fp_0,\fp_1,\fp_\infty)$, $(\fp'_0,\fp'_1,\fp'_\infty)$ be two
triples of mutually complementary points determining the same circle $C\subset
M$.
\begin{enumerate}
\item There is a unique linear fractional transformation $s(t)$ such
that
\begin{equation*}
\Gamma_{\fp_0}^{\fp_\infty}(t)\fp_1=\Gamma_{\fp'_0}^{\fp'_\infty}(s(t))\fp'_1,
\end{equation*}
for all $t\in\R\cup\set{\infty}$.  In particular,
$\Gamma_{\fp_0}^{\fp_\infty}\fp_1$ is determined by its values at any
three distinct points of $\R\cup\set{\infty}$.
\item
$\cross(\fp_1,\fp_\infty,\Gamma_{\fp_0}^{\fp_\infty}(t)\fp_1,\fp_0)=t$,
for all $t\in\R\cup\set{\infty}$.
\end{enumerate}
\end{prop}

\begin{rem}\label{th:21}
If $\fg$ is a complex Lie algebra so that $M$ is a conjugacy class of
complex parabolic subalgebras, the entire discussion goes through
unchanged with $t\in\C$.  Thus, in this case, pairwise complementary
$\fp_0,\fp_1,\fp_\infty$ determine a rational curve $C$ in $M$ with
$C\cong\C P^1$ carrying a complex projective structure and
well-defined complex cross-ratio.  Moreover, Proposition~\ref{th:10}
is still valid so long as $t$ is taken to be complex.  We will
persist in calling $C$ the circle through $\fp_0,\fp_1,\fp_\infty$.
\end{rem}

With this in hand, we conclude this section with a lemma whose
unpromising statement is tailored for an application in
\MySec\ref{sec:dress-pairs-curv}.
\begin{lemm}
\label{th:22}
Let $\fg$ be complex and $M$ a self-dual symmetric $R$-space for
$\fg$.  Let $\fp,\fp_1,\fp_2\in M$ with $(\fp,\fp_1)$ complementary and
set $\tau=\Gamma_{\fp}^{\fp_1}(-1)\in\Aut(\fg)$.  Finally, let
$w\in\C^\times$ and put $\fr=\Gamma_{\fp}^{\fp_1}(w)\fp_2$.

Then $\fp,\fp_1,\fp_2$ are pairwise complementary if and only if
$\fr,\tau\fr$ are complementary.  In this case, all five parabolic
subalgebras are concircular.
\end{lemm}
\begin{proof}
First suppose that $\fp,\fp_1,\fp_2$ are pairwise complementary then
\begin{equation*}
\tau\fr=\Gamma_\fp^{\fp_1}(-1)\Gamma_{\fp}^{\fp_1}(w)\fp_2=\Gamma_{\fp}^{\fp_1}(-w)\fp_2.
\end{equation*}
Thus, since $w\neq -w$, $\fr$ and $\tau\fr$ are distinct points on
the circle through $\fp,\fp_1,\fp_2$.

For the converse, we must work a little harder.  So suppose that
$\fr,\tau\fr$ are complementary and let $\xi=\xi_{\fr}^{\tau\fr}$ be
the corresponding grading element.  Clearly we have $\tau\xi=-\xi$
and we use this to see that $\fr$ (and so $\tau\fr$) is complementary
to both $\fp$ and $\fp_1$.  For this, let
$x\in(\fp^\perp\oplus\fp_1^\perp)\cap\fr^\perp$ so that, on the one
hand, $\tau x=-x$ and, on the other hand, $[\xi, x]=-x$. Then
\begin{equation*}
x=-\tau x=\tau[\xi,x]=[-\xi,-x]=-x
\end{equation*}
so that $x=0$.  Thus
$\fp^\perp\cap\fr^\perp=\set{0}=\fp_1^\perp\cap\fr^\perp$ so that
both $(\fp,\fr)$ and $(\fp_1,\fr)$ are complementary pairs.  We now
have $\fp_2=\Gamma_\fp^{\fp_1}(1/w)\fr$ so that $\fp_2$ lies on the
circle through $\fp,\fp_1,\fr$ and, since $1/w\neq 0,\infty$, is
distinct from and so complementary to $\fp$ and $\fp_1$.
\end{proof}

\subsection{Bianchi permutability for Darboux transforms}
\label{sec:bianchi-perm-darb}
Let $(f,\eta):\Sigma\to M$ be an isothermic map to a self-dual
symmetric $R$-space $M$ with Darboux transforms $(f_1,\eta_1)$,
$(f_2,\eta_2)$ with parameters $m_1\neq m_2$ respectively:
$f_1=\darb_{m_1}f$, $f_2=\darb_{m_2}f$.  Suppose, in addition, that
the $f_i$ are pointwise complementary.  We seek a fourth isothermic
map $\hat f$ so that $\hat f=\darb_{m_2}f_1=\darb_{m_1}f_2$.  

Recall from Theorem~\ref{th:4} that, for $i=1,2$ we have
$\Gamma_f^{f_i}(1-\frac{t}{m_i})\cdot(\D+t\eta)=\D+t\eta_i$.  In
particular, since $f_2$ is $(\D+m_2\eta)$-parallel,
$f_{12}=\Gamma_f^{f_1}(1-\frac{m_2}{m_1})f_2$ is
$(\D+m_2\eta_1)$-parallel.  Moreover $(f_1,f_{12})$ is the image by
$\Gamma_f^{f_1}(1-\frac{m_2}{m_1})$ of the complementary pair
$(f_1,f_2)$ and so is a complementary pair also.  Thus
$(f_{12},\eta_{12})$ is a Darboux transform of $(f_1,\eta_1)$ is
parameter $m_2$ and $\eta_{12}$ is determined by the requirement that
$\Gamma_{f_1}^{f_{12}}(1-\frac{t}{m_2})\cdot(\D+t\eta_1)=\D+t\eta_{12}$.

Similarly $(f_{21},\eta_{21})$ is a Darboux transform of
$(f_2,\eta_2)$ with parameter $m_1$ where
$f_{21}=\Gamma_f^{f_2}(1-\frac{m_1}{m_2})f_1$ and $\eta_{21}$ is
determined by the requirement that
$\Gamma_{f_2}^{f_{21}}(1-\frac{t}{m_1})\cdot(\D+t\eta_2)=\D+t\eta_{21}$.
It therefore suffices to prove that $f_{12}=f_{21}$ and
$\eta_{12}=\eta_{21}$.  

For the first of these identities, note that
$\Gamma_{f}^{f_{1}}(t)f_2$ and $\Gamma_{f}^{f_{2}}(t)f_1$ parametrise
the circles through corresponding points of $f, f_1, f_2$ and so
differ by a linear fractional transformation of $t$ by
Proposition~\ref{th:10}.  This linear fractional transformation must
fix $0$ and swop $1$ and $\infty$ so that, for all $t$,
\begin{equation*}
\Gamma_{f}^{f_{1}}(t)f_2=\Gamma_{f}^{f_{2}}(\tfrac{t}{t-1})f_1
\end{equation*}
and evaluating this at $t=1-m_2/m_1$ yields $f_{12}=f_{21}$.

The second identity is equivalent to $\D+t\eta_{12}=\D+t\eta_{21}$ and
since these connections are gauges of $\D+t\eta$ by
$\Gamma_{f_1}^{\hat
f}(1-\frac{t}{m_2})\Gamma_f^{f_1}(1-\frac{t}{m_1})$ and
$\Gamma_{f_2}^{\hat
f}(1-\frac{t}{m_1})\Gamma_f^{f_2}(1-\frac{t}{m_2})$, it suffices
to prove:
\begin{lemm}
\label{th:11}
$\Gamma_{f_1}^{\hat
f}(1-\frac{t}{m_2})\Gamma_f^{f_1}(1-\frac{t}{m_1})=
\Gamma_{f_2}^{\hat
f}(1-\frac{t}{m_1})\Gamma_f^{f_2}(1-\frac{t}{m_2})=
\Gamma^{f_1}_{f_2}(\frac{1-t/m_{1}}{1-t/m_{2}})$, for all $t$.
\end{lemm}
\begin{proof}
It suffices to prove that $\Gamma_{f_1}^{\hat
f}(1-\frac{t}{m_2})\Gamma_f^{f_1}(1-\frac{t}{m_1})=
\Gamma^{f_1}_{f_2}(\frac{1-t/m_1}{1-t/m_2})$ for then the other
equality follows by swopping the roles of $f_1$ and $f_2$.
With $\Gamma(t)$ denoting $\Gamma_{f_1}^{\hat
f}(1-\frac{t}{m_2})\Gamma_f^{f_1}(1-\frac{t}{m_1})$, we observe
that, by definition, $\Gamma(t)$ acts as
$\frac{1-t/m_{1}}{1-t/m_2}$ on $f_1^\perp$ and as $1$ on
$f_1/f_1^\perp$.  Moreover, $\Gamma(t)$ preserves $f_2$: indeed, this
amounts to demanding that $\Gamma^{f_1}_{\hat
f}(1-\frac{t}{m_2})f_2=\Gamma_f^{f_1}(1-\frac{t}{m_1})f_2$.  We
readily check that these agree at $t=0,m_1,\infty$ and so everywhere
by Proposition~\ref{th:10}.

Now, since $\Gamma(t)$ is orthogonal with respect to the Killing
form, we see that it must preserve the decomposition
$\ul\fg=f_1^\perp\oplus(f_1\cap f_2)\oplus f_2^\perp$ and so acts as
$1$ on $f_1\cap f_2$.  Moreover, since $f_1^\perp$ and $f_2^\perp$ are
dual with respect to the Killing form, $\Gamma(t)$ acts as
$\frac{1-t/m_2}{1-t/m_1}$ on $f_2^\perp$ and the lemma follows.
\end{proof}

Finally, apply the identity of Lemma~\ref{th:11} to $f$ to get
\begin{equation*}
\Gamma^{f_1}_{f_2}(\tfrac{1-t/m_{1}}{1-t/m_{2}})f=
\Gamma_{f_1}^{\hat f}(1-\tfrac{t}{m_2})\Gamma^{f_1}_f(1-\tfrac{t}{m_1})f=
\Gamma_{f_1}^{\hat f}(1-\tfrac{t}{m_2})f
\end{equation*}
and evaluate at $t=\infty$ to conclude:
\begin{equation*}
\hat f=\Gamma^{f_1}_{f_2}(m_2/m_1)f
\end{equation*}
whence $\cross(f,f_1,\hat f,f_2)=m_2/m_1$.  To summarise:
\begin{thm}
\label{th:12}
Let $f_1,f_2:\Sigma\to M$ be complementary Darboux transforms of an
isothermic map $f:\Sigma\to M$ with parameters $m_1\neq m_2$
respectively.  Then there is a common Darboux transform $\hat
f=\darb_{m_1}f_2=\darb_{m_2}f_1$ of $f_1$ and $f_2$ which is
pointwise concircular with $f,f_1,f_2$ and has constant cross-ratio
$m_2/m_1$ with these.
\end{thm}
We call such a configuration of four isothermic maps a \emph{Bianchi
quadrilateral} (see Figure~\ref{fig:quad}).

\begin{figure}[htp]
\begin{center}
\includegraphics{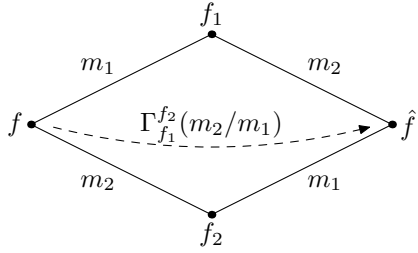}
  \caption{A Bianchi quadrilateral\label{fig:quad}}
\end{center}
\end{figure}

\subsection{The Cube Theorem}
\label{sec:cube-theorem}

With almost no extra effort, we can compute the effect of a third
Darboux transform on a Bianchi quadrilateral and so prove the
analogue of Bianchi's Cube Theorem \cite[\S11]{Bianchi1905a} in our
context.

So start with an isothermic map $f$, distinct
$m_1,m_2,m_3\in\R^\times$ and pairwise complementary Darboux
transforms $f_1,f_2,f_3$ of $f$: $f_i=\darb_{m_i}f$.  Then
Theorem~\ref{th:12} provides three simultaneous Darboux transforms
$f_{ij}=f_{ji}=\darb_{m_i}f_j=\darb_{m_j}f_i$, for $i,j\in\set{1,2,3}$
distinct.   Suppose now that the $f_{ij}$ are also pairwise
complementary and appeal once more to Theorem~\ref{th:12} to
construct three more Bianchi quadrilaterals $(f_i,
f_{ij},f_{i(jk)},f_{ik})$ so that
$f_{i(jk)}=\darb_{m_k}f_{ij}=\darb_{m_j}f_{ik}$.  The claim is that
these three new maps $f_{i(jk)}$ coincide (as do the corresponding
$\eta_{i(jk)}$) so that we have eight isothermic maps arranged into
six Bianchi quadrilaterals with the combinatorics of a cube (see
Figure~\ref{fig:cube}).

\begin{figure}[htp]
\begin{center}
\includegraphics{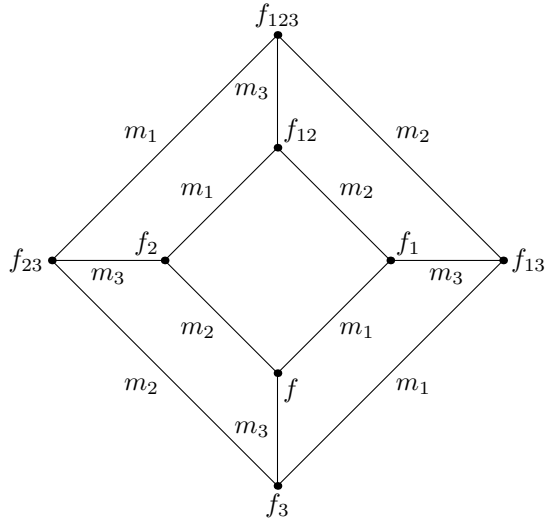}
  \caption{A Bianchi cube\label{fig:cube}}
\end{center}
\end{figure}

To prove this, it is enough to show that $f_{1(23)}=f_{2(31)}$ and
that $\eta_{1(23)}=\eta_{2(31)}$ for the rest follows by permuting
indices. Now
\begin{align*}
f_{1(23)}&=\Gamma_{f_1}^{f_{12}}(1-\tfrac{m_3}{m_2})f_{13}\\
&=\Gamma_{f_1}^{f_{12}}(1-\tfrac{m_3}{m_2})\Gamma_{f}^{f_1}(1-\tfrac{m_3}{m_1})f_3\\
\intertext{and similarly}
f_{2(31)}&=\Gamma_{f_2}^{f_{12}}(1-\tfrac{m_3}{m_1})\Gamma_{f}^{f_2}(1-\tfrac{m_3}{m_2})f_3.
\end{align*}
The identity of Lemma~\ref{th:11} evaluated at $t=m_3$ now yields
\begin{equation*}
f_{1(23)}=f_{2(31)}=\Gamma_{f_2}^{f_1}(\tfrac{1-{m_3}/{m_1}}{1-{m_3}/{m_2}})f_3
\end{equation*}
and we learn a little more: the eighth surface $f_{123}$ is
concircular with $f_1, f_2,f_3$ with constant cross-ratio:
\begin{equation*}
\cross(f_3,f_1,f_{123},f_2)=\tfrac{1-{m_3}/{m_1}}{1-{m_3}/{m_2}}.
\end{equation*}
For the classical case of isothermic surfaces in $S^n$, this last
statement is due to Bobenko--Suris \cite{Bobenko2002}.

Finally, we have
$\D+t\eta_{1(23)}=\Gamma_{f_{12}}^{f_{123}}(1-\tfrac{t}{m_3})\cdot(\D+t\eta_{12})=\D+t\eta_{2(31)}$
so that $\eta_{1(23)}=\eta_{2(31)}$.

To summarise: 
\begin{thm}
\label{th:13}
Let $f:\Sigma\to M$ be an isothermic map and
$m_1,m_2,m_3\in\R^\times$ distinct parameters.  Let $f_1,f_2,f_3$ be
pairwise complementary Darboux transforms of $f$, $f_i=\darb_{m_i}f$,
and $f_{12},f_{23},f_{13}$ pairwise complementary simultaneous
Darboux transforms of the $f_i$ as in Theorem~\ref{th:12}:
$f_{ij}=\darb_{m_i}f_j=\darb_{m_j}f_i$.  Then there is an isothermic
map $(f_{123},\eta_{123})$ which is a simultaneous Darboux transform
of each $f_{ij}$: $f_{123}=\darb_{m_k}f_{ij}$, for $i,j,k$ distinct.

Moreover, $f_{123},f_1,f_2,f_3$ are concircular with cross-ratio
$\frac{1/m_3-1/m_1}{1/m_3-1/m_2}$.
\end{thm}

\begin{rem}
The cube configuration is highly symmetrical and, in particular,
starting the analysis at $f_{123}$ rather than $f$, one readily
deduces that the other tetrahedron in the cube has concircular
vertices with the same cross-ratio:
\begin{equation*}
\cross(f_{12},f_{23},f,f_{13})=\tfrac{1-m_3/m_2}{1-m_3/m_1}.
\end{equation*}
\end{rem}

\subsection{Permutability of Christoffel and Darboux transforms}
\label{sec:perm-christ-darb}

Theorem~\ref{th:27} realises the Christoffel transform of $f$ as a
limit of conjugated Darboux transforms.  Taken together with the
preceding permutability theorems for Darboux transforms, this allows
us to prove the analogue of two more permutability results of
Bianchi.

First we prove that Christoffel and Darboux transforms commute
(c.f.~Bianchi \cite[\S3]{Bianchi1905a}):
\begin{thm}
\label{th:28} 
Let $f,f^c,\hat{f}:\Sigma\to M$ be isothermic maps into a self-dual
symmetric $R$-space with $(f^{c},\eta^{c})$ a Christoffel transform
of $f$ and $(\hat{f},\hat{\eta})$ a Darboux transform of $f$ with
parameter $m$.  

Then there is a fourth isothermic map $\hat{f}^c:\Sigma\to M$ which
is simultaneously a Christoffel transform of $\hat{f}$ and a Darboux
transform of $f^c$ with parameter $m$: $\hat{f}^{c}=(\hat{f})^c=\darb_mf^c$.
\end{thm}
\begin{proof}
Let $f^{c}$ be a Christoffel transform with respect to
$(\fp_{0},\fp_{\infty})$ and find, via Theorem~\ref{th:27}, a family
$({f}_s,{\eta}_{s})$ of Darboux transforms of $f$ with
parameter $s$ such that
\begin{equation*}
\lim_{s\to 0}{f}_s=\fp_{\infty},\qquad
\lim_{s\to
0}\Gamma_{\fp_0}^{\fp_{\infty}}(-s)\act{f}_s=f^{c},\qquad
\lim_{s\to
0}\Gamma_{\fp_0}^{\fp_{\infty}}(-s){\eta}_{s}=\eta^{c}.
\end{equation*}
Now apply Theorem~\ref{th:12} to $f_{s},\hat{f}$ to get
$(\hat{f}_{s},\hat{\eta}_{s})$ with
$\hat{f}_{s}=\darb_{m}f_s=\darb_{s}\hat{f}$.  We have
$\hat{f}_s=\Gamma_f^{\hat{f}}(1-s/m)f_s$ so that 
\begin{equation*}
\lim_{s\to 0}\hat{f}_s=\Gamma_f^{\hat{f}}(1)\fp_{\infty}=\fp_{\infty}
\end{equation*}
so that Theorem~\ref{th:27} applies to the $\hat{f}_s$ and we may
define a Christoffel transform $(\hat{f}^c,\hat{\eta}^c)$ of
$\hat{f}$ by
\begin{equation*}
\hat{f}^c=\lim_{s\to0}\Gamma_{\fp_0}^{\fp_{\infty}}(-s)\act\hat{f}_s,\qquad
\hat{\eta}^c=\lim_{s\to0}\Gamma_{\fp_0}^{\fp_{\infty}}(-s)\hat{\eta}_s.
\end{equation*}
It remains to show that $(\hat{f}^c,\hat{\eta}^c)$ is a Darboux
transform with parameter $m$ of $f^{c}$.  By Proposition~\ref{th:24},
this is the case if and only if
$\Gamma_{f^c}^{\hat{f}^c}(1-\tfrac{t}{m})\cdot(\D+t\eta^{c})=\D+t\hat{\eta}^{c}$
for all $t\neq m$.  However, we already know that, for each $s$,
$\Gamma_{f_{s}}^{\hat{f}_s}(1-\tfrac{t}{m})\cdot(\D+t\eta_s)=\D+t\hat{\eta}_s$
and applying $\Gamma_{\fp_0}^{\fp_{\infty}}(-s)$ to both sides yields
\begin{equation*}
\Gamma_{\Gamma_{\fp_0}^{\fp_{\infty}}(-s)f^{c}}^{\Gamma_{\fp_0}^{\fp_{\infty}}(-s)\hat{f}_{s}}(1-\tfrac{t}{m})
\Gamma_{\fp_0}^{\fp_{\infty}}(-s)\cdot(\D+t\eta_{s})=
\Gamma_{\fp_0}^{\fp_{\infty}}(-s)\cdot(\D+t\hat\eta_{s}).
\end{equation*}
Now let $s\to 0$ and use the last assertion of Theorem~\ref{th:27} to
get
$\Gamma_{f^c}^{\hat{f}^c}(1-\tfrac{t}{m})\cdot(\D+t\eta^{c})=\D+t\hat{\eta}^{c}$
as required.
\end{proof}

Similarly, we can take $m_{3}=s$ and let $s\to 0$ in the Cube
Theorem~\ref{th:13} to obtain, without further argument, the
following result due to Bianchi \cite[\S9]{Bianchi1905a} in the
classical case:
\begin{thm}
\label{th:29}
Let $f,f_{1},f_{2},f_{12}:\Sigma\to M$ be isothermic maps in a
self-dual symmetric $R$-space forming a Bianchi quadrilateral.  Let
$f^{c}$ be a Christoffel transform of $f$.

Then there are Christoffel transforms $f_{1}^{c},f_2^c,f_{12}^c$ of
$f_{1},f_{2},f_{12}$ respectively so that
$f^{c},f_{1}^{c},f_2^c,f_{12}^c$ also form a Bianchi quadrilateral.
\end{thm}

\subsection{Discrete isothermic surfaces in self-dual symmetric $R$-spaces}
\label{sec:discr-isoth-surf}

It is an experimental fact \cite{BobenkoPinkall96,BobenkoPinkall96a}
(with some theoretical underpinning \cite{BobenkoSuris2005}) that
the combinatorics of the B\"acklund transforms of an integrable
system provide an integrable discretisation of that system.  In
particular, such an analysis is available for isothermic surfaces in
the conformal sphere \cite{BobenkoPinkall96a,Jeromin2000}.

We now indicate how the same ideas may be applied to give a satisfying
theory of discrete isothermic nets in self-dual symmetric $R$-spaces
which replicates essentially all features of the smooth theory we
have been developing.  

For this we begin with a subset $\Omega\subset\Z^2$ of a
$2$-dimensional lattice.  If $i,j,k,l\in\Omega$ are the vertices of an
elementary quadrilateral as in  Figure~\ref{fig:discrete}, we denote the
oriented edge from $i$ to $j$ by $(j,i)$ and the quadrilateral by $(i,j,k,l)$.

\begin{figure}[htp]
\begin{center}
\includegraphics{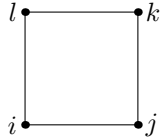}
  \caption{An elementary quadrilateral\label{fig:discrete}}
\end{center}
\end{figure}

Thanks to the discussion in \MySec\ref{sec:circles-self-dual},
Hertrich-Jeromin's definition of an isothermic net in the conformal
sphere \cite{Jeromin2000} carries straight over into our setting.
\begin{defn}
Let $f:\Omega\to M$ be a map into a self-dual symmetric $R$-space.

We say that $f$ is an \emph{isothermic net} if, for each
elementary quadrilateral $(i,j,k,l)$, the values
$f(i),f(j),f(k),f(l)$ are concircular and pairwise complementary
with
\begin{equation*}
\cross(f(i),f(j),f(k),f(l))={m(i,l)}/{m(i,j)},
\end{equation*}
for some \emph{factorising function} $m$ which is a real valued
function on the unoriented edges of $\Omega$ with equal values on
opposite edges\footnote{Thus $m(i,j)=m(j,i)$ for all edges and
$m(i,j)=m(l,k)$, $m(i,l)=m(j,k)$ on elementary quadrilaterals.}.
\end{defn}

Just as in the smooth case, isothermic nets have a zero-curvature
representation: let $(f,m)$ be such a net and let $V=\Omega\times\fg$
be the trivial $\fg$-bundle.  For each oriented edge $(j,i)$ and
$t\in\R$, define $\Gamma^{t}(j,i):V_{i}\to V_{j}$ by
\begin{equation*}
\Gamma^t(j,i)=\Gamma^{f(j)}_{f(i)}(1-\tfrac{t}{m(i,j)}).
\end{equation*}
It is immediate that $\Gamma^0(j,i)=1$ while
$\Gamma^t(j,i)\Gamma^t(i,j)=1_{V(j)}$, so that each $\Gamma^t$ is
(the holonomy of) a discrete connection on the bundle $V$.  A mild
reorganisation of the discussion leading to the proof of
Theorem~\ref{th:12} now gives:
\begin{thm}
$(f,m)$ is an isothermic net if and only if each $\Gamma^t$ is a flat
connection: that is, on each elementary quadrilateral $(i,j,k,l)$, we have
\begin{equation*}
\Gamma^t(k,j)\Gamma^t(j,i)=\Gamma^t(k,l)\Gamma^t(l,i).
\end{equation*}
\end{thm}
(See \cite{Bobenko2002,BurHerRos08} for this gauge-theoretic approach
to isothermic nets in the conformal case.) 

With this in hand, the transformation theory of isothermic nets can
be developed along the same lines as the smooth case.  Firstly,
$T$-transforms arise by trivialising the family of connections
$\Gamma^{t}$: locally we may find $\Phi_{t}:\Omega\to\Aut(\fg)$,
unique up to left multiplication by constants in $\Aut(\fg)$, such
that
\begin{equation*}
\Gamma^t(j,i)=\Phi_{t}(j)^{-1}\Phi_t(i)
\end{equation*}
and then define the $T$-transforms $\ttrans_{s}f$ of $f$ by
$\ttrans_{s}f=\Phi_{s}f$.  One can show:
\begin{thm}
Let $(f,m)$ be an isothermic net.  Then each $T$-transform
$\ttrans_{s}f$ is an isothermic net with factorising function $m-s$.
\end{thm}
Indeed, an easy calculation shows that the connection ${}^s\Gamma^t$
associated to $(\ttrans_{s}f,m-s)$ is the gauge
transform\footnote{Thus
${}^s\Gamma^t(j,i)=\Phi_{s}(j)\Gamma^{t+s}(j,i)\Phi_s(i)^{-1}$, for
all edges $(j,i)$.} by $\Phi_s$ of $\Gamma^{t+s}$ and so is flat.

Again, Darboux transforms amount to parallel bundles of parabolic
subalgebras: a map $\hat f:\Omega\to M$ is a \emph{Darboux transform
of $f$ with parameter $\hat m\in\R$} if $\hat f$ is $\Gamma^{\hat
m}$-parallel.  Thus, for all edges $(i,j)$,
\begin{equation*}
\Gamma^{\hat m}(j,i)\hat f(i)=\hat f(j).
\end{equation*}
It follows from the definition that, for each edge $(j,i)$, the
points $f(i),f(j),\hat f(j),\hat f(i)$ are concircular with
cross-ratio $\hat m/m(i,j)$ and then Lemma~\ref{th:11} applies to
show that the connections $\hat\Gamma^{t}$ of $(\hat f,m)$ are the
gauges of $\Gamma^{t}$ by $i\mapsto \Gamma_{f(i)}^{\hat
f(i)}(1-\tfrac{t}{\hat m})$ so that: 
\begin{thm}
$\hat f$ is isothermic with the same factorising function $m$.
\end{thm}
Moreover, the argument of Theorem~\ref{th:13} now gives
\begin{equation*}
\hat f(k)=\Gamma^{f(j)}_{f(l)}\left(\tfrac{1-\hat m/m(i,j)}{1-\hat m/m(i,l)}\right)\hat f(i)
\end{equation*}
so that $\hat f(i),f(j),\hat f(k),f(l)$ are concircular with
\begin{equation*}
\cross(\hat f(i),f(j),\hat f(k),f(l))=\tfrac{1-\hat m/m(i,j)}{1-\hat m/m(i,l)}.
\end{equation*}
In particular, $\hat f(k)$ depends only on $\hat f(i),f(j)$ and
$f(l)$ which is the \emph{tetrahedron property} of Bobenko--Suris
\cite{Bobenko2002}.

One can go further and construct Christoffel transforms via blow-ups
of Darboux transforms or (with a little more work) via
stereoprojections and then see that all our permutability theorems
relating these transforms hold in the discrete setting (with
essentially the same proofs).  However, all this would take us too
far afield for the present.  We may return to these topics elsewhere.

\section{Curved flats and Darboux pairs}\label{sec:curved}

The main result of \cite{Burstall1997a} asserts that Darboux pairs
of isothermic surfaces in $S^3$ are precisely the curved flats in the
symmetric space of point-pairs $S^3\times S^3\setminus\Delta$.  This
was extended in \cite{Burstall2004} to isothermic surfaces in $S^n$
where the dressing transformation of curved flats was related to
Darboux transforms of Christoffel pairs.

We now show that this circle of ideas carries through in our general
context of isothermic maps to symmetric $R$-spaces.  We begin by
rehearsing the relevant details of the theory of curved flats.

\subsection{Curved flats and their transformations}
\label{sec:curved-flats}
Let $N$ be a symmetric $G$-space\footnote{We retain our standing
assumption that $G$ is the adjoint group of $\fg$ but much of the
theory of curved flats carries through without it.}: thus each $x\in N$ has stabiliser
$H_x$ open in the fixed set of an involution $\tau_x\in\Aut(G)$.  The
derivative of $\tau_x$, also called $\tau_x\in\Aut(\fg)$, has
eigenvalues $\pm1$ with eigenspaces $\fh_x,\fm_x$ where $\fh_x$ is
the Lie algebra of $H_x$ and $T_xN\cong\fm_{x}$ via the solder form $\beta^N_x$.

\begin{defn}[\cite{Ferus1996}]
A \emph{curved flat} in a symmetric space $N$ is a map
$\phi:\Sigma\to N$ such that each $\phi^*\beta^N(T_p\Sigma)$ is an
abelian subalgebra of $\fm_{\phi(p)}$, $p\in\Sigma$.
\end{defn}

Curved flats have a gauge-theoretic formulation that will be basic
for us.  Let $\phi:\Sigma\to N$ be a map and contemplate the section
$\tau$ of $\Aut(\ul\fg)$ given by
$\tau(p)=\tau_{\phi(p)}$.  There is a corresponding
eigenbundle decomposition $\ul\fg=\phi^{-1}\fh\oplus\phi^{-1}\fm$
with $\phi^{-1}\fm\cong\phi^{-1}TN$ and thus a decomposition of the
flat connection $\D$:
\begin{equation}
\label{eq:17}
\D=\cD+\cN
\end{equation}
where $\cD\tau=0$ while $\cN\in\Omega^1_\Sigma(\phi^{-1}\fm)$
(explicitly, $\cN=-\half\tau\D\tau$).  From
\MySec\ref{sec:homog-geom-sold}, we have $\cN=\phi^*\beta^N$ so the
curved flat condition is precisely
\begin{equation*}
[\cN\wedge\cN]=0. 
\end{equation*}

On the other hand, the flatness of $\D$ yields
\begin{equation*}
0=R^\D=R^\cD+\D^\cD\cN+\half[\cN\wedge\cN]
\end{equation*}
the $\phi^{-1}\fh$- and $\phi^{-1}\fm$-components of which are,
respectively, the Gauss and Codazzi equations of the situation:
\begin{subequations}\label{eq:13}
\begin{align}
\label{eq:14}
0&=R^\cD+\half[\cN\wedge\cN]\\
\label{eq:15}
0&=\D^\cD\cN. 
\end{align}
\end{subequations}
We therefore conclude:
\begin{prop}
\label{th:14}
For a map $\phi:\Sigma\to N$, the following are equivalent:
\begin{enumerate}
\item $\phi$ is a curved flat;
\item $\cD$ is a flat connection;
\item $[\cN\wedge\cN]=0$. 
\end{enumerate}
\end{prop}

\subsubsection{Zero curvature representation and spectral deformation}
\label{sec:spectral-deformation}

Curved flats comprise an integrable system which, in the case where each
$\phi^{*}\beta^N(T_p\Sigma)$ is maximal abelian semisimple, is
gauge-equivalent to the ``$G/H$-system'' introduced by Terng
\cite{Terng1997,Bruck2002}.  At the root of all this is the
observation that, just like isothermic maps, curved flats are
characterised by the flatness of a pencil of connections.  Indeed, for
$u\in\R$, define $G$-connections $\D^u$ on $\ul\fg$ by
\begin{equation}\label{eq:16}
\D^u=\cD+u\cN
\end{equation}
so that, in particular, $\D^0=\cD$ and $\D^1=\D$. 

\begin{prop}[c.f.~\cite{Ferus1996}]\label{th:15}
$\phi:\Sigma\to N$ is a curved flat if and only if $\D^u$ is flat
for all $u\in\R$. 
\end{prop}
\begin{proof}
The coefficients of $u$ in   $R^{\D^u}$ are $R^\cD$, $\D^\cD\cN$ and
$\half[\cN\wedge\cN]$ of which $\D^\cD\cN$ vanishes automatically
from the Codazzi equation \eqref{eq:15} while, by
Proposition~\ref{th:14}, the others vanish exactly if $\phi$ is a
curved flat. 
\end{proof}

We may now argue as in \MySec\ref{sec:zero-curv-repr} to see that
curved flats come in $1$-parameter families: for $u\in\R$, we can
trivialise $\D^u$: locally we can find gauge transformations
$\Psi_u$, unique up to left multiplication by a constant element of
$G$, with $\Psi_u\cdot\D^u=\D$ .  We have:
\begin{prop}\label{th:16}
$\phi_u:=\Psi_u\phi:\dom(\Psi_u)\subset\Sigma\to N$ is a curved
flat. 

Moreover, for $u,v\in\R$, $(\phi_u)_v=\phi_{uv}$ modulo the action of $G$. 
\end{prop}
\begin{proof}
We have $\Psi_u(\phi^{-1}\fh)=\phi_u^{-1}\fh$,
$\Psi_u(\phi^{-1}\fm)=\phi_u^{-1}\fm$ so that in the decomposition of
$\D$ with respect to $\ul\fg=\phi_u^{-1}\fh\oplus\phi_u^{-1}\fm$:
\begin{equation*}
\D=\cD^u+\cN^u
\end{equation*}
we have $\cD^u=\Psi_u\cdot\cD$, $\cN^u=\Psi_u\cN$ and thus 
$\Psi_u\cdot(\cD+uv\cN)=\cD^u+v\cN^u$.  In particular
$\cD^u+v\cN^u$ is flat for all $v\in\R$ so that $\phi_u$ is a curved
flat by Proposition~\ref{th:15}. 

Further, with $\Psi_v^u\cdot(\cD^u+v\cN^u)=\D$, we have
$\Psi_{uv}=\Psi_v^u\Psi_u$ whence $\phi_{uv}=(\phi_u)_v$, up to the
action of $G$. 
\end{proof}

We call $\set{\phi_u:u\in\R}$ the \emph{associated family} or
\emph{spectral deformation} of $\phi$.  We note that $\phi_0$ is
constant.

\subsubsection{Dressing transformations}
\label{sec:dress-transf}

There is another class of transformations of curved flats which
shares some of the structure of the Darboux transformations discussed
above in that the initial data is a parallel bundle of height one
parabolic subalgebras with respect to a certain connection.  However,
for curved flats, the construction of the new curved flat from this
data is somewhat more elaborate. 

We begin with a curved flat $\phi:\Sigma\to N$ with associated field
of involutions $\tau\in\Gamma\Aut(\fg)$ and pencil of flat
connections $\D^u$, $u\in\R$.  We extend the definition of $\D^u$ by
taking $u\in\C$ to get a holomorphic pencil of $G^\C$-connections on
$\ul\fg{}^\C$.  We observe that this pencil is uniquely characterised
by the following properties:
\begin{enumerate}[1.] 
\item $u\to\D^u$ is holomorphic on $\C$ with a simple pole at $\infty$;\label{item:4}
\item $\tau\circ\D^u\circ\tau^{-1}=\D^{-u}$, for $u\in\C$;\label{item:5}
\item $\overline{\D^u}=\D^{\bar{u}}$, for $u\in\C$;\label{item:6}
\item $\D^1=\D$.\label{item:7}
\end{enumerate}

Now let $w\in\C^\times\setminus\set{\pm 1}$ with $w^2\in\R$ and suppose we are given a bundle of
height one parabolic subalgebras $\fr\leq\ul\fg{}^\C$ with the
following properties:
\begin{enumerate}[(a)]
\item $\fr$ is $\D^w$-parallel;\label{item:1}
\item $\fr$ and $\tau\fr$ are complementary;\label{item:2}
\item $\fr=\bar{\fr}$ if $w\in\R$ and $\fr=\tau\bar{\fr}$ if $w\in i\R$.\label{item:3}
\end{enumerate}

\begin{rem}
The local existence of such $\fr$ is a point-wise affair: given a
height one parabolic subalgebra $\fr_0\leq\fg$ and $p\in\Sigma$ with
$\fr_0,\tau(p)\fr_0$ complementary and satisfying condition
(\ref{item:3}), one can define $\fr$ by parallel transport since $\D^u$
is flat and preserves condition (\ref{item:3}).  However, the existence
of such a $\fr_0$ imposes strong restrictions on the symmetric space
$N$.  We have already noted that the existence of any height one
parabolic subalgebra $\fr_{0}\leq\fg^{\C}$ excludes $\fg=\fg_{2}$,
$\mathfrak{f}_4$ or $\mathfrak{e}_8$ and conditions (\ref{item:2})
and (\ref{item:3}) further restrict the possibilities for $\tau(p)$.
Thus, for example, for $N=G_k(\K^n)$ a Grassmannian for
$\K=\R,\C,\HH$, viewed as a Riemannian symmetric space, such $\fr_0$
are only available when $n=2k$.
\end{rem}

Given such a bundle $\fr$, we define the family of gauge
transformations
$\Gamma(u):=\Gamma^\fr_{\tau\fr}(\tfrac{u-w}{u+w})\in\Gamma\Aut(\ul\fg{}^\C)$
and observe that $\Gamma(u)$ enjoys similar properties to $\D^u$:
\begin{enumerate}[i.] 
\item $u\mapsto \Gamma(u)$ is holomorphic on $\C P^1$ except for
simple poles at $\pm w$;
\item $\tau\Gamma(u)\tau^{-1}=\Gamma(-u)$, for $u\in\C$;
\item $\overline{\Gamma(u)}=\Gamma(\bar u)$. 
\end{enumerate}
Now contemplate the family of connections
\begin{equation*}
\hat{\D}^u=\Gamma(u)\cdot \D^u. 
\end{equation*}
We readily check that they share properties \ref{item:5} and
\ref{item:6} with $\D^u$ and that $u\to\hat{\D}^u$ is holomorphic on
$\C$ except possibly at $\pm w$.  In fact, the singularities there
are removable:
\begin{lemm}\label{th:17}
$u\mapsto\hat{\D}^u$ is holomorphic near $\pm w$. 
\end{lemm}
\begin{proof}
We show that $u\mapsto\hat{\D}^u$ is holomorphic near $w$ and then
appeal to the symmetry $u\mapsto -u$.  Contemplate the eigenbundle
decomposition of $\ad\xi^\fr_{\tau\fr}$:
\begin{equation*}
\ul\fg=\fr^\perp\oplus(\fr\cap\tau\fr)\oplus \tau\fr^\perp
\end{equation*}
and the corresponding decomposition of $\D^w$:
\begin{equation*}
\D^w=D-b-\hat b
\end{equation*}
with $D\xi^\fr_{\tau\fr}=0$, $b\in\Omega^1(\fr^\perp)$ and
$\hat b\in\Omega^1(\tau\fr^\perp)$.  Since $\fr$ is $\D^w$-parallel
and preserved by both $D$ and $b$, we conclude that $\hat b$ takes
values in $\fr\cap\tau\fr^\perp$ and so vanishes.  Thus
\begin{equation*}
\D^u=D-b+(u-w)\cN
\end{equation*}
whence
\begin{align*}
\hat{\D}^u&=\Gamma(u)\cdot D -\Gamma(u)b+(u-w)\Gamma(u)\cN\\
&=D-\tfrac{u-w}{u+w}b+(u-w)\Gamma(u)\cN,
\end{align*}
which last is clearly holomorphic near $w$. 
\end{proof}
Thus $u\mapsto\hat{\D}^u$ is holomorphic on $\C$ and has a simple
pole at $\infty$, since $\Gamma$ is holomorphic and
$\Aut(\fg)$-valued there.  

Finally, set $\hat\phi:=\Gamma(1)^{-1}\phi:\Sigma\to N$ with
corresponding field of involutions
$\hat\tau:=\Gamma(1)^{-1}\tau\Gamma(1)$ (we assume\footnote{This
assumption is justified for our applications in
\MySec\ref{sec:curved-flats-z}.}, if necessary, that $ G^\C\cap
\Aut(\fg)$ acts on $N$).  Now $u\mapsto \Gamma(1)^{-1}\cdot\hat\D^u$
has properties \ref{item:4}--\ref{item:7} of $\D^u$ with $\hat\tau$
replacing $\tau$ so we conclude that this is the holomorphic pencil
of connections associated with $\hat\phi$.  Since these connections
are flat, we deduce that $\hat\phi$ is a curved flat.  To summarise:
\begin{prop}
\label{th:18}
Let $\phi:\Sigma\to N$ be a curved flat and $\fr\leq\ul\fg$ a bundle
of height one parabolic subalgebras with properties
(\ref{item:1})--(\ref{item:3}) above for some
$w\in\C^\times\setminus\set{\pm 1}$ with $w^2\in\R$.  Then
$\hat\phi:\Sigma\to N$ defined by
\begin{equation*}
\hat\phi=\Gamma_{\tau\fr}^\fr(\tfrac{1+w}{1-w})\phi
\end{equation*}
is also a curved flat.  We call $\hat\phi$ a \emph{dressing
transform} of $\phi$. 
\end{prop}

\subsubsection{Relationship with loop group formalism}
\label{sec:relat-with-loop}

This construction is essentially an invariant\footnote{Thus without
the use of frames.} reformulation of the Terng--Uhlenbeck
construction \cite{Terng2000} of \emph{dressing by simple factors}. 
We digress to give a brief account of this. 

In the loop group formalism, one works with frames of the underlying
maps into homogeneous spaces.  In the case at hand, this means we fix
a base-point $o\in N$ and, for $\phi:\Sigma\to N$, contemplate maps
$F:\Sigma\to G$ with $\phi=F\act o$.  Given such a frame $F$, let
$\alpha=F^{-1}\D F\in\Omega_\Sigma^1(\fg)$ and write
$\alpha=\alpha_{\fh_o}+\alpha_{\fm_o}$ for the decomposition of
$\alpha$ according to the decomposition of $\fg$ into eigenspaces of
$\tau_o$:
\begin{equation*}
\fg=\fh_o\oplus\fm_o. 
\end{equation*}
Viewing $F$ as a gauge transformation, we see that
$F\cdot(\D+\alpha)=\D$ and that $F$ intertwines the decompositions
$\ul\fg=\ul\fh_o\oplus\ul\fm_o$ and
$\ul\fg=(\phi^{-1}\fh)\oplus(\phi^{-1}\fm)$.  It follows at once that
$F\cdot(\D+\alpha_{\fh_o})=\cD$ and $F\act\alpha_{\fm_o}=\cN$ so
that $\phi$ is a curved flat if and only if
$[\alpha_{\fm_o}\wedge\alpha_{\fm_o}]=0$. 

In this case, for $u\in\C$, set
$\alpha_u=\alpha_{\fh_o}+u\alpha_{\fm_o}$ and note that
$\D+\alpha_u=F^{-1}\cdot\D^u$ and so is flat.  We may therefore
integrate, at least locally, to find $F_u:\Sigma\to G^\C$ with
$F_u^{-1}\D F_u=\alpha_u$ or, equivalently,
$F_u\cdot(\D+\alpha_u)=\D$.  The family $\set{F_u:u\in\C}$
constitute an \emph{extended frame} of $\phi$. 

Observe that $(F_u F^{-1})\cdot\D^u=\D$ so that $\Psi_u=F_u F^{-1}$
giving
\begin{equation*}
\phi_u=\Psi_u\phi=F_u\act o
\end{equation*}
so that $F_u$ is a frame of the spectral deformation $\phi_u$.

The philosophy of the loop group formalism is to view the maps $F_u$
as a single map into a loop group on which there is a (local) action
of a second loop group via Birkhoff factorisation.  In the case at
hand, let $\cG_+$ denote the group (under pointwise multiplication)
of holomorphic maps $g:\C\to G^\C$ which are \emph{real} in that $\cl
g(u)=g(\cl u)$ and \emph{twisted} in that $\tau_o g(u)=g(-u)$.  If
the constants of integration are chosen correctly, each $u\mapsto
F_u(p)$, $p\in\Sigma$, is an element of $\cG_+$.  Now let $\cG_-$ be
the group of real, twisted rational maps $g$ on $\C P^1$ with values
in $G^\C$ which are holomorphic near $\infty$ with $g(\infty)=1$.  It
follows from the Birkhoff decomposition theorem \cite{Pressley1986}
that a generic real, twisted $G^\C$-valued holomorphic map $g$ on $\C
P^1\setminus D$, $D$ a divisor, has a unique factorisation
\begin{equation*}
g=g_+g_-
\end{equation*}
with $g_\pm\in\cG_\pm$.  This leads to a local action of $\cG_-$
on $\cG_+$ by
\begin{equation*}
g_-\# g_+=(g_-g_+)_+
\end{equation*}
which can be shown to preserve the set of extended frames. 

In general, the Birkhoff factorisation is non-local and so difficult to
compute but the key observation of Terng and Uhlenbeck
\cite{Terng2000,Uhlenbeck1989} is that, for certain basic elements of
$\cG_-$, the \emph{simple factors}, the factorisation can be carried
out explicitly (see \cite[\S4.3]{Burstall2004} for a conceptual
discussion of this).  For our setting, we define the simple factors
as follows: let $\fr_o\leq\fg^\C$ be a parabolic subalgebra of height
one and $w\in\C^\times\setminus\set{\pm1}$ such that
$(\fr_0,\tau_o\fr_o)$ are a complementary pair and $\fr=\cl\fr_0$,
for $w\in\R$ and $\fr=\tau_o\cl\fr_0$ otherwise.  Our simple factor
is now defined to be $u\mapsto
\Gamma_o(u):=\Gamma^{\fr_o}_{\tau_o\fr_o}(\tfrac{u-w}{u+w})$.  One
can show (see \cite[Proposition~4.11]{Burstall2004} for the case
where $\fg=\fso(n+1,1)$ and $\fr_o$ is the stabiliser of a null line
in $(\R^{n+1,1})^\C$) that here the Birkhoff factorisation is given
by
\begin{equation*}
\Gamma_o g_+=\hat{g}_+\hat\Gamma
\end{equation*}
with
$\hat\Gamma(u)=\Gamma_{\tau_o\hat\fr}^{\hat\fr}(\tfrac{u-w}{u+w})$
and $\hat\fr=g_+(w)^{-1}\fr_o$.  We apply this to the extended frame
$u\mapsto F_u$ to get a new extended frame $\hat F_u$ which we
evaluate at $u=1$ to get a new curved flat $\hat\phi$.  Thus
\begin{align*}
\hat\phi&=\Gamma^{\fr_o}_{\tau_o\fr_o}(\tfrac{1-w}{1+w})F
\bigl(\Gamma^{\hat\fr}_{\tau_o\hat\fr}(\tfrac{1-w}{1+w})\bigr)^{-1}\act
o\\
&=\Gamma^{\fr_o}_{\tau_o\fr_o}(\tfrac{1-w}{1+w})F
\bigl(\Gamma^{\hat\fr}_{\tau_o\hat\fr}(\tfrac{1-w}{1+w})\bigr)^{-1}F^{-1}\act\phi\\
&=\Gamma^{\fr_o}_{\tau_o\fr_o}(\tfrac{1-w}{1+w})
\Gamma^{F\hat\fr}_{F\tau_o\hat\fr}(\tfrac{1+w}{1-w})\act\phi. 
\end{align*}
Now set $\fr=F\hat\fr=FF_w^{-1}\fr_0=\Psi_w^{-1}\fr_0$ which is
$\D^u$-parallel since $\fr_o$ is constant so $\D$-parallel.  Then
$F\tau_o\hat\fr=\tau\fr$ since $\tau=F\tau_o F^{-1}$ and we conclude
that $\hat\phi$ yielded by the dressing action coincides with that of
Proposition~\ref{th:18} up to the (irrelevant) constant factor
$\Gamma^{\fr_o}_{\tau_o\fr_o}(\tfrac{1-w}{1+w})$.

\subsection{Curved flats in $Z$ are Darboux pairs}
\label{sec:curved-flats-z}

We now apply this theory to the case where the symmetric space is the
space $Z\subset M\times M^*$ of complementary pairs of parabolic
subalgebras in dual symmetric $R$-spaces $M$ and $M^*$. 

A map $\phi=(f,\hat f):\Sigma\to Z$ can be viewed as a pair of maps
$f:\Sigma\to M$, $\hat f:\Sigma\to M^*$ which are pointwise
complementary.  Our first result is that $\phi$ is a curved flat if
and only if $f$ and $\hat f$ are a Darboux pair of isothermic
maps.  For this, first note that the field of involutions $\tau$
along $\phi$ is given by $\tau=\Gamma_f^{\hat f}(-1)$ so that 
\begin{equation*}
\phi^{-1}\fh=f\cap\hat f,\qquad\phi^{-1}\fm=f^{\perp}\oplus\hat f^{\perp}. 
\end{equation*}
Thus the decomposition \eqref{eq:4} of $\D$ coincides with the
decomposition \eqref{eq:17} yielding $\cN=-\beta-\hat\beta$.  With
this in mind, we revisit the argument that established Theorem
\ref{th:4}: we have
\begin{equation}\label{eq:19}
\Gamma_f^{\hat f}(1/u)\cdot\D^u=\cD-\beta-u^2\hat\beta
=\D+m(1-u^2)\hat\beta/m=\nabla^{m(1-u^2)}
\end{equation}
and, similarly,
\begin{equation}
\label{eq:20}
\Gamma_f^{\hat f}(1/u)\cdot\D^u=\D+m(1-u^2)\beta/m=\hat\nabla^{m(1-u^2)}.
\end{equation}
Thus, if $\phi=(f,\hat f)$ is a curved flat, each $\D^u$ is flat
whence each $\nabla^{m(1-u^2)},\hat\nabla^{m(1-u^2)}$ is flat so that
both $(f,\hat\beta/m)$ and $(\hat f,\beta/m)$ are isothermic.
Clearly $\hat f$ is $\nabla^m$-parallel while $f$ is
$\hat\nabla^m$-parallel and we conclude that they are $m$-Darboux
transforms of each other.\footnote{Note that $m\neq 0$ may be chosen
arbitrarily.}

Conversely, if $(f,\eta)$ is isothermic and $\hat f=\darb_m f$, we
know that $\eta=\hat\beta/m$ so that flatness of $\nabla^t$, for all
$t$, forces flatness of $\D^u$, for all $u\neq 0$ and so, by
continuity, flatness of $\D^0$ also whence $\phi= (f,\hat f)$ is a
curved flat.   We have therefore proved:
\begin{thm}
\label{th:19}
A map $\phi=(f,\hat f):\Sigma\to Z$ is a curved flat if and only if
$f,\hat f$ are a Darboux pair of isothermic maps.
\end{thm}

\subsubsection{Spectral deformation is $T$-transform}
\label{sec:spectr-deform-t}

It is now straightforward to relate the spectral deformation of
curved flats in $Z$ to $T$-transforms of the constituent isothermic
maps.  Indeed, the spectral deformation of $\phi=(f,\hat f)$ is given by
$\phi_u=\Psi_u\phi$ where $\Psi\cdot\D^u=\D$ while $\ttrans_t
f=\Phi_t f$, $\ttrans_t\hat f=\hat\Phi_t\hat f$ where
$\Phi_t\cdot\nabla^t=\D$ and $\hat\Phi_t\cdot\hat\nabla^t=\D$.  It
follows at once from \eqref{eq:19} and \eqref{eq:20} that we may take
\begin{equation*}
\Psi_u=\Phi_{m(1-u^2)}\circ\Gamma_{f}^{\hat f}(1/u)=
\hat\Phi_{m(1-u^2)}\circ\Gamma_{f}^{\hat f}(u).
\end{equation*}
Since $\Gamma_f^{\hat f}$ preserves both $f$ and $\hat f$ we conclude:
\begin{thm}
\label{th:20}
Let $f,\hat f:\Sigma\to M$ be isothermic with $\hat f=\darb_m f$ and
let $\phi_u=(f_u,\hat f_u):\Sigma\to Z$ be a spectral deformation of
the curved flat $\phi=(f,\hat f)$.  Then, for $u\neq 0$,
\begin{equation*}
f_u=\ttrans_{m(1-u^2)}f\qquad\hat f_u=\ttrans_{m(1-u^2)}\hat f.
\end{equation*}
\end{thm}

\subsubsection{Dressing pairs of curved flats are Bianchi quadrilaterals}
\label{sec:dress-pairs-curv}

Suppose now that $M$ is self-dual and recall the discussion of
Bianchi permutability in \MySec\ref{sec:bianchi-perm-darb}: given
$(f,\eta):\Sigma\to M$ isothermic with pointwise complementary
Darboux transforms $f_i=\darb_{m_i}f$, $m_1\neq m_2$, there is a
fourth isothermic map $\hat f$ given by
\begin{equation*}
\hat f=\Gamma^{f_1}_{f_2}(m_2/m_1)f=\Gamma_{f}^{f_1}(1-\tfrac{m_2}{m_1})f_2=
\Gamma_f^{f_2}(1-\tfrac{m_1}{m_2})f_1
\end{equation*}
such that $\hat f=\darb_{m_1}f_2=\darb_{m_2}f_1$.  We claim that the
curved flat $(f_2,\hat f)$ is a dressing transform of $(f,f_1)$ and
that all dressing pairs of curved flats $\Sigma\to Z\subset M\times
M$ arise this way.

For this, let $\phi=(f,f_1):\Sigma\to Z$ with
$\tau=\Gamma_f^{f_1}(-1)$ the field of involutions along $\phi$ and
$\D^u$ the pencil of flat connections.  Choose
$w\in\C^\times\setminus\set{\pm 1}$ such that $m_2=m_1(1-w^2)$ and
note that $\fr:=\Gamma_f^{f_1}(w)f_2^\C$ is $\D^w$-parallel if and
only if $f_2$ is $\nabla^{m_2}$-parallel since
$\Gamma_f^{f_1}(w)\cdot\nabla^{m_2}=\D^w$ by virtue of \eqref{eq:19}.
Moreover, Lemma~\ref{th:22} tells us that $f,f_1,f_2$ are pairwise
complementary at each point if and only if $\fr,\tau\fr$ are
pointwise complementary.  Finally,
\begin{equation*}
\cl{\fr}=\Gamma_f^{f_1}(\cl{w})\cl{f_2^\C}
\end{equation*}
so that $f_2^\C=\cl{f_2^\C}$ if and only if $\cl{\fr}=\fr$ or
$\tau\fr$ according to the sign of $w^2$ since
$\Gamma_f^{f_1}(-w)=\tau\Gamma_f^{f_1}(w)$.  Thus $f,f_1,f_2$ define
a Bianchi quadrilateral of isothermic maps if and only if $\fr$
satisfies the conditions to define a dressing transformation. We now
inspect the effect of that dressing transformation: from Proposition
\ref{th:10}, we have, for all $\lambda\in\C\cup\set{\infty}$,
\begin{align}
\Gamma_{\tau\fr}^{\fr}(\tfrac{w+\lambda}{w-\lambda})f&=
\Gamma_f^{f_1}(\lambda)f_2\label{eq:21}\\
\Gamma_{\tau\fr}^{\fr}(\tfrac{\lambda+w}{\lambda-w})f_1&=
\Gamma_f^{f_1}(\lambda)f_2\label{eq:22}
\end{align}
since both sides of \eqref{eq:21} agree at $\lambda=0,\pm w$ while
both sides of \eqref{eq:22} agree at $\lambda=\infty,\pm w$.  Now
evaluate \eqref{eq:21} at $\lambda=w^2=1-m_2/m_1$ and \eqref{eq:22} at
$\lambda=1$ to conclude that
\begin{equation*}
\Gamma_{\tau\fr}^{\fr}(\tfrac{1+w}{1-w})(f,f_1)=(\hat f,f_2)
\end{equation*}
as required.

In summary:
\begin{thm}
\label{th:23}
Let $M$ be a self-dual symmetric $R$-space and $Z\subset M\times M$
the space of complementary pairs.  Let $\phi=(f,f_1):\Sigma\to Z$ and
$\hat\phi=(\hat f, f_2):\Sigma\to Z$ be Darboux pairs of isothermic
maps with $f_1=\darb_{m_1}f$.

Then $\hat\phi$ is a dressing transform of $\phi$ with data
$\fr:\Sigma\to M^\C$ and $w\in\C$ if and only if $f,f_1,f_2,\hat f$
form a Bianchi quadrilateral.  Moreover, in this case,
$\fr=\Gamma_f^{f_1}(w)f_2^\C$ and $f_2=\darb_{m_2}f$ where
$m_2=m_1(1-w^2)$.
\end{thm}

\begin{rem}
What can be said when $M$ is not self-dual?  Darboux pairs of
isothermic maps are still curved flats and there may still be
dressing transformations available although the $\fr$ must
necessarily take values in a self-dual complex $R$-space as $\fr$ and
$\tau\fr$ are $ G^\C$-conjugate (via $\tau$).  An example of this
situation is provided by Darboux pairs of isothermic maps into
Grassmannians $f:\Sigma\to G_k(\R^{2n})$, $f_1:\Sigma\to G_{n-k}(\R^{2n})$,
$k<n$, with $\fr:\Sigma\to G_n(\C^{2n})$.  In such a case, the
dressing transformation will provide a new Darboux pair but its
relation with the original pair requires further investigation.  We
may return to this elsewhere.
\end{rem}

\section{Nondegenerate isothermic submanifolds }\label{sec:examples}

\subsection{A quadratic form}
\label{sec:quadratic-form}

Let $(f,\eta):\Sigma\to M$ be an isothermic map to a symmetric
$R$-space.  Recall from \MySec\ref{sec:defin-zero-curv} that we may view
$\eta$ as an $f^{-1}T^{*}M$-valued $1$-form and so define a
$2$-tensor $q_f$ on $\Sigma$ by contracting $\eta$ with $\D f$:
$q_{f}(X,Y)=\eta_X(\D f_{Y})$.  We are about to see that $q_f$ is
symmetric and so we call it the \emph{quadratic form associated to
$(f,\eta)$}.

This quadratic form is invariant under all the transformations of
isothermic maps we have discussed:

\begin{prop}
\label{th:30}
Let $(f,\eta):\Sigma\to M$ be an isothermic map with associated
quadratic form $q$.  Then
\begin{enumerate}
\item $q$ is symmetric.\label{item:8}
\item Any Christoffel, Darboux or $T$-transform of $(f,\eta)$ also
has associated quadratic form $q$.\label{item:9}
\item Let $\phi=(f,\hat f):\Sigma\to Z\subset M\times M^{*}$ be a
Darboux pair with $\hat f=\darb_m f$.  Then
\begin{equation*}
q=\tfrac{1}{2m}\phi^{*}g_{Z},
\end{equation*}
where $g_{Z}$ is the (neutral signature) metric on $Z$ induced by the
Killing form.\label{item:10}
\end{enumerate}
\end{prop}
\begin{proof}
To get a convenient formulation of $q$, we pull back the
soldering isomorphism of \MySec\ref{sec:homog-geom-sold} to view $\D f$
as a $1$-form with values in $\ul\fg/f$ so that $q=(\eta,\D f)$.

Now equation \eqref{eq:18} pulls back to give
\begin{equation*}
\D s\equiv [\D f,s]\mod f,
\end{equation*}
for all $s\in\Gamma f$, and, since the fibres of $f$ are
self-normalising, $\D f$ is uniquely characterised by this property.

With this in hand, let $\hat f=\darb_{m}f$ be a Darboux transform of
$f$ and recall the decomposition $\ul\fg=f^{\perp}\oplus (f\cap {\hat
f})\oplus{\hat f}^{\perp}$ into eigenbundles of $\xi_f^{\hat{f}}$
with the accompanying decomposition of connections \eqref{eq:4}:
\begin{equation*}
\D=\cD-\beta-\hat\beta.
\end{equation*}
We see immediately that
\begin{align*}
m\hat\eta&=\beta\equiv\D f\mod f &
m\eta&=\hat\beta\equiv\D \hat{f}\mod\hat{f}
\end{align*}
so that $q(X,Y)=\frac{1}{m}(\hat\beta_{X},\beta_{Y})=q_{\hat f}(Y,X)$.

On the other hand, the $f\cap\hat{f}$-component of $\D\eta$ is
$-[\beta\wedge\eta]$ so that $[\beta\wedge\hat\beta]$ vanishes and we
have
\begin{align*}
q(X,Y)=-([\xi_f^{\hat f},\hat\beta_{X}],\beta_{Y})&=-(\xi_f^{\hat
f},[\hat\beta_{X},\beta_{Y}])\\
&=-(\xi_f^{\hat f},[\hat\beta_{Y},\beta_{X}])=q(Y,X).
\end{align*}
This settles the symmetry of $q$ and yields $q=q_{\hat f}$ also.
Moreover, with $\phi=(f,\hat f):\Sigma\to Z$, we have
\begin{equation*}
\phi^{*}g_{Z}=(\cN,\cN)
\end{equation*}
where $\cN=-\beta-\hat\beta$ is the pull-back by $\phi$ of the solder
form of $Z$ from which assertion \ref{item:10} follows.

It remains to treat Christoffel and $T$-transforms of $f$.  For the
first of these, let $(f^c,\eta^{c})$ be a Christoffel transform of
$(f,\eta)$ with respect to $(\fp_0,\fp_{\infty})$ with corresponding
stereoprojections $F,F^{c}$.  Thus $f=\exp F\act\fp_{0}$ and
$\eta=\exp F\act\D F^{c}$.  By writing $s\in\Gamma f$ as $s=\exp
F\act\sigma$ for $\sigma:\Sigma\to\fp_{0}$ and differentiating, one
readily checks that 
\begin{equation*}
\D F=\exp F \D F\equiv \D f\mod f
\end{equation*}
so that $q(X,Y)=(\D F^{c}_{X},\D F_{Y})=q_{f^c}(Y,X)$.  We therefore
deduce from the symmetry of $q$ that $q=q_{f^c}$.

Finally let $(f_{t},\eta_{t})=(\Phi_{t}\act f,\Phi_t\act\eta)$ be a $T$-transform of $(f,\eta)$ where
$\Phi_{t}\cdot(\D+t\eta)=\D$.  We observe that, for $s\in\Gamma f$,
\begin{equation*}
(\D+t\eta)s\equiv \D s\mod f
\end{equation*}
from which it follows that $\D f_{t}=\Phi_{t}\act \D f$ and thus,
since $\Phi_{t}$ is isometric for the Killing form, $q=q_{f_t}$.
\end{proof}

\subsection{Nondegenerate isothermic submanifolds}
\label{sec:nondegen-isoth}

\begin{defn}
We say that an isothermic map $(f,\eta)$ is \emph{nondegenerate} if
its associated quadratic form $q_f$ is nondegenerate.
\end{defn}

Clearly, in this case $f$ immerses so that $(f,\eta)$ is an
isothermic submanifold and we note from
Proposition~\ref{th:30}(\ref{item:9}) that all transforms of
$(f,\eta)$ are nondegenerate also.

\begin{rem}
For the classical case of isothermic surfaces in $S^n$, $q$ is, in
fact, a holomorphic quadratic differential (see, for example,
\cite[Lemma~2.1]{Burstall2004}) and so is nondegenerate off a
divisor.
\end{rem}

\subsubsection{Dimension bounds}
\label{sec:dimension-bounds}

The dimension of a nondegenerate isothermic
submanifold of a symmetric $R$-space $M$ is bounded by the rank of
the associated symmetric space $Z\subset M\times M^{*}$ of
complementary pairs.

To recall what is involved in this, we begin with a (not necessarily
Riemannian) symmetric $G$-space $N$ and the associated symmetric
decomposition
\begin{equation*}
\fg=\fh_x\oplus\fm_{x},
\end{equation*}
for some $x\in N$.  A \emph{Cartan subspace} of $\fm_{x}$ is a
maximal abelian subspace $\fa$ of $\fm_{x}$ all of whose elements
are semisimple.  It is known that there are a finite number of $\Ad
H_{x}$ conjugacy classes of these \cite[Theorem~3]{OshimaMatsuki1980}
and they all have the same dimension
\cite[page~14]{FlenstedJensen1986}: this is the \emph{rank} of $N$.

With this understood, we have:
\begin{thm}
\label{th:31}
Let $(f,\eta):\Sigma\to M$ be a nondegenerate isothermic submanifold
and $Z\subset M\times M^{*}$ the symmetric space of complementary
pairs.  Then
\begin{equation}
\label{eq:23}
\dim\Sigma\leq\rank Z.
\end{equation}
\end{thm}

To prove this, let $\hat f=\darb_{m}f$ be a Darboux transform of $f$
and contemplate the curved flat $\phi=(f,\hat f):\Sigma\to Z$.  From
Proposition~\ref{th:30}(\ref{item:10}), $q_{f}$ coincides up to scale
with the metric induced by $\phi$ so that each
$\phi^{*}\beta^{Z}(T_{p}\Sigma)$ is an abelian subspace of
$\fm_{\phi(p)}$ on which the Killing form of $\fg$ is nondegenerate.
Our result therefore follows from:
\begin{prop}\label{th:34}
Let $U\subset\fm_{x}$, $x\in Z$ be an abelian subspace on which the
Killing form of $\fg$ is nondegenerate.  Then $\dim U\leq\rank Z$
and equality holds if and only if $U$ is a Cartan subspace of
$\fm_{x}$.
\end{prop}
\begin{proof}
Let $W\subset\fm_{x}$ be a maximal abelian subspace of $\fm_x$
containing $U$.  We argue as in Carleson--Toledo
\cite[Lemma~4.2]{CarTol89}: recall that any $X\in\fg$ has a unique
Jordan decomposition $X=X_s+X_n$ with $[X_s,X_n]=0$, $\ad X_s$
semisimple and $\ad X_n$ nilpotent.  Moreover both $\ad X_{s}, \ad
X_{n}$ are polynomials without constant term in $\ad X$.  The
uniqueness of the decomposition yields $X_s,X_n\in\fm_x$ whenever
$X\in\fm_{x}$ while it follows from the maximal abelian property of
$W$, that, for $X\in W$, we have $X_{s},X_{n}\in W$ also.  Thus we
may write $W=W_{s}\oplus W_{n}$ where $W_{s},W_{n}$ consist
respectively of the elements of $W$ which are semisimple,
respectively, nilpotent (these are linear subspaces of $W$ since the
sum of commuting semisimples is semisimple and similarly for
nilpotents).  Moreover, for $X\in W_{n}$ and $Y\in W$, $\ad X\circ
\ad Y$ is nilpotent so that the Killing inner product
$(X,Y)=\trace(\ad X\circ \ad Y)=0$ and we have $W_{n}\subset W\cap
W^{\perp}$.  Thus, since the Killing form is nondegenerate on $U$, we
must have $U\cap W_{n}=\set{0}$ so that $\dim U\leq\dim W_{s}$.

Now choose a maximal toral subspace $\fa$ of $\fm_{x}$ with
$W_{s}\subset\fa$.  Thus $\fa$ is a (necessarily abelian) subspace
in $\fm_{x}$ all of whose elements are semisimple and maximal for
this last property.  Lepowsky--McCollum \cite[Corollary to
Theorem~5.2]{LepowskyMcCollum1976} prove that $\fa$ is then
maximal abelian also and so a Cartan subspace of $\fm_x$.  We
therefore have
\begin{equation*}
\dim U\leq \dim W_s \leq\dim\fa=\rank Z.
\end{equation*}
Moreover, in the case of equality, $W_{s}=\fa$ is maximal abelian so
that $W_{n}=\set{0}$ and then $U=\fa$ also.
\end{proof}

We can draw a geometric corollary of this development:
\begin{cor}
\label{th:33}
Let $(f,\eta):\Sigma\to M$ be a nondegenerate isothermic submanifold
of maximal dimension: $\dim\Sigma=\rank Z$.  Then the associated
quadratic form $q_f$ is a flat pseudo-Riemannian metric on $\Sigma$.
\end{cor}
\begin{proof}
Let $\hat f$ be a Darboux transform of $f$ and again contemplate
$\phi=(f,\hat f):\Sigma\to Z$.  We know that $q_{f}$ is, up to scale,
the metric induced on $\Sigma$ by $\phi$ while
Proposition~\ref{th:34} tells us that the soldering image of each
$\D\phi(T_pM)$ is a Cartan subspace.  That the metric induced by
$\phi$ is flat is now a result of Ferus--Pedit \cite[Theorem~2]{Ferus1996} (see also
Remark~1 of that paper).
\end{proof}

\begin{rem}
Corollary~\ref{th:33} provides distinguished coordinates on
nondegenerate isothermic submanifolds.  For isothermic surfaces in
$S^{n}$, these include the conformal curvature line coordinates that
provided the original definition of an isothermic surface
\cite{Bour1862,Cay72}.  It seems likely that similar coordinates may be
available, at least for self-dual $M$, which are related to the
\emph{generalised conformal structure} defined on $M$ by
Gindikin--Kaneyuki \cite{GinKan98}.  We may return to this elsewhere.
\end{rem}

For $\fg$ simple, we can readily compute $\rank Z$ in terms of more
familiar invariants.  To do this, we define the rank of a symmetric
$R$-space $M$ to be the rank of $M$ when viewed as a compact
Riemannian symmetric space with isometry group a maximal compact
subgroup $K$ of $G$.  In more detail, the symmetric decomposition of
$\fk$ at $\fp\in M$ is given by
\begin{equation*}
\fk=(\fp\cap\ci\fp\cap\fk)\oplus(\fp^{\perp}\oplus\ci\fp^{\perp})\cap\fk,
\end{equation*}
where $\ci$ is the Cartan involution fixing $\fk$, and the rank of
$M$ is the dimension of a maximal abelian subspace of the second
summand.

We now have:
\begin{prop}
Let $\fg$ be simple and $M$ be a symmetric $R$-space for $G$ with
$Z\subset M\times M^{*}$ the symmetric space of complementary pairs.
\begin{enumerate}
\item If $\fg$ is complex so that $M$ is a Hermitian symmetric
$K$-space, then $\rank Z=2\rank M$.
\item Otherwise $\rank Z=\rank M^{\C}$.
\end{enumerate}
\end{prop}
\begin{proof}
We begin with $\fg$ complex so that $\fg=\fk^{\C}$ with $\fk$ a
compact simple Lie algebra and the corresponding Cartan involution is
just complex conjugation across $\fk$.  At $x=(\fp,\cl{\fp})\in Z$,
we have $\fm_{x}=\fp^{\perp}\oplus\cl{\fp}^{\perp}=(\fm_{x}\cap\fk)^{\C}$.
Let $\fa$ be a Cartan subspace of $\fm_x$.  Since the Lie bracket on
$\fg$ is complex linear, $\fa$ is necessarily a complex subspace.
Now, let $\fc$ be a Cartan subspace of $\fm_{x}\cap\fk$.  Then
$\fc^{\C}$ is another Cartan subspace of $\fm_{x}$ and, by a theorem
of Kostant--Rallis \cite[Theorem~1]{Kostant1971}, all such Cartan
subspaces are $G^{x}$-conjugate and so have the same dimension.  Thus
\begin{equation*}
\rank Z=\dim_{\R}\fa=2\dim_{\C}\fa=2\dim_{\C}\fc^{\C}=2\dim_{\R}\fc=2\rank M.
\end{equation*}
Now consider the case where $\fg$ is not complex.  Let $(\fp,\fq)\in
Z$ and $\fa\subset \fp^{\perp}\oplus\fq^{\perp}=\fm_{(\fp,\fq)}$ a
Cartan subspace so that $\rank Z=\dim\fa$.  Then it is easy to see
that $\fa^{\C}\subset\fm_{(\fp,\fq)}^{\C}=\fm_{(\fp^{\C},\fq^{\C})}$ is
a Cartan subspace at $(\fp^{\C},\fq^{\C})\in Z^{\C}\subset M^{\C}\times
(M^{\C})^{*}$.  Thus, the first part of the proposition applied to
$M^{\C}$ yields
\begin{equation*}
2\rank Z=\dim_{\R}\fa^{\C}=\rank Z^{\C}=2\rank M^{\C}
\end{equation*}
whence the result.
\end{proof}
The outcome of this analysis is shown in Table~\ref{tab:r-space} on
page~\pageref{tab:r-space}.

\subsubsection{Existence}
\label{sec:existence}

We now show that the dimension bounds of the last section are sharp:
for any symmetric $R$-space $M$, we can find a nondegenerate
isothermic submanifold of maximal dimension and so, by applying the
transformation theory of \MySec\ref{sec:isoth-subm}, infinitely many
such.  

For this, let $(\fp_0,\fp_{\infty})\in M\times
M^{*}$ be a complementary pair and choose a Cartan subspace $\fa\subset
\fp_0^{\perp}\oplus\fp_{\infty}^{\perp}$.  Define
$F:\fa\to\fp_{\infty}^{\perp}$ and $F^{c}:\fa\to\fp_0^{\perp}$ to be
the restrictions to $\fa$ of the projections onto
$\fp_{\infty}^{\perp},\fp_{0}^{\perp}$ along
$\fp_{0}^{\perp},\fp_{\infty}^{\perp}$, respectively.  Then both
$F,F^c$ are injective immersions ($\fa$ has no intersection with
either $\fp_{0}^{\perp}$ or $\fp_{\infty}^{\perp}$ since all its
elements are semisimple).  For any $X\in T_{H}\fa$, we have $\D
F_{H}(X)+\D F_{H}^c(X)=F(X)+F^c(X)=X$ so that, since
$\fa,\fp_{0}^{\perp},\fp_{\infty}^{\perp}$ are all abelian, we have
\begin{equation*}
[\D F\wedge\D F^{c}]=0.
\end{equation*}
Thus $(F,F^{c})$ is the stereoprojection of a Christoffel pair of
isothermic submanifolds.  Otherwise said, define $f:\fa\to M$ and
$\eta\in\Omega^{1}_{\fa}(f/f^{\perp})$ by
\begin{equation*}
f=\exp(F)\act \fp_0,\qquad\eta=\exp(F)\D F^{c}
\end{equation*}
and conclude that $(f,\eta)$ is an isothermic submanifold.  Moreover,
in the proof of Proposition~\ref{th:30}, we saw that $q_{f}(X,Y)=(\D
F_{X},\D F^{c}_{Y})=(\D F_{Y},\D F^{c}_{X})$ but this last is just
$\half(X,Y)$ so that $q_{f}$ is nondegenerate.

We can also write down an explicit gauge transformation relating the
flat connections $\Dt=\D+t\eta$ to the trivial connection so that the
computation of iterated Darboux and $T$-transforms of $(f,\eta)$ is a
purely algebraic matter.  Indeed, $\Dt=\exp(F)\cdot(\D+\D F+t\D
F^{c})$ and we observe that $F+tF^{c}$ takes values in a fixed
abelian (in fact, Cartan, for $t\neq 0$) subspace so that
\begin{equation*}
\D=\exp(F+ tF^{c})\cdot(\D+\D F+t\D F^{c})=\exp(F+t F^{c})\exp(-F)\cdot\Dt
\end{equation*}

To summarise:
\begin{thm}
Let $M$ be a symmetric $R$-space, $(\fp_0,\fp_{\infty})\in M\times
M^{*}$ be a complementary pair and $\fa\subset
\fp_0^{\perp}\oplus\fp_{\infty}^{\perp}$ a Cartan subspace.

Then $(f,\eta):\fa\to M$ defined as above is a nondegenerate
isothermic submanifold of maximal dimension with
\begin{equation*}
\exp(F+t F^{c})\exp(-F)\cdot(\D+t\eta)=\D.
\end{equation*}
\end{thm}

\subsection{Examples}
\label{sec:examples-1}

The main motivating example for our theory is the case when $M$ is
the projectivisation of some real quadric where we recover the rich
theory of isothermic surfaces in conformal spheres of arbitrary
signature.  We conclude our account by briefly contemplating what is
known for some other classical symmetric $R$-spaces.

\subsubsection{Projective spaces and higher flows}
\label{sec:projective-spaces}

It follows from the results of \MySec\ref{sec:dimension-bounds} that a
nondegenerate isothermic submanifold of $\R P^n$ is necessarily a
curve.  Moreover any curve $f$ in $\R P^n$ is isothermic with respect
to any $\eta\in\Omega^{1}(f^{-1}T^{*}\R P^{n})$.  In
particular, in sharp contrast to the conformal case, $\eta$ is not
uniquely determined by $f$.

Similarly, a non-degenerate isothermic submanifold of $\C P^{n}$ of
maximal dimension is a surface and, in fact, it is not difficult to
see that it must be a holomorphic curve.  Again, any holomorphic
curve $f$ in $\C P^{n}$ is isothermic with respect to any holomorphic
$\eta\in\Omega^{1}(f^{-1}T^{*}\C P^{n})$.

However, even the case $n=1$ is not completely banal if we consider
the \emph{dynamics} of isothermic curves.  For this, contemplate a
non-degenerate isothermic curve $(f,\eta):\Sigma\to\R P^{1}$, viewed as a line
subbundle of a trivial $\R^2$-bundle over a $1$-manifold $\Sigma$.
The $1$-form $\eta$ determines a coordinate $x$ on $\Sigma$ (unique
up to sign and translations) for which $q_{f}=\D x^{2}$: thus, for
any $\psi\in\Gamma f$,
\begin{equation*}
\eta(\partial/\partial x)\psi_{x}=\psi.
\end{equation*}
We normalise $\psi$ so that $\psi\wedge\psi_x$ is a constant section
of $\Wedge^2\ul\R^2$.  It follows that $\psi\wedge\psi_{xx}=0$ so
that
\begin{equation}
\label{eq:25}
\psi_{xx}=p\psi
\end{equation}
where $p$, so defined, is the \emph{projective curvature} of $f$.

Following Pinkall \cite{Pinkall1995} (see also
\cite{Burstall2002b,CaliniIvey2008,GoldsteinPetrich1991}) we suppose
now that $f$ evolves so that
\begin{equation}
\label{eq:26}
f_{t}=p f_{x}.
\end{equation}
Thus we have a map $f:\Sigma\times I\to \R P^{1}$ and we again
contemplate the normalised lift $\psi\in\ul{\R^{2}}_{\Sigma\times I}$
with $\psi\wedge\psi_{x}$ constant.  This last, together with
\eqref{eq:26}, yields
\begin{equation*}
\psi_{t}=-\frac{p_{x}}{2}\psi+p\psi_{x}
\end{equation*}
and then $\psi_{xxt}=\psi_{xtx}$ yields the KdV equation:
\begin{equation}
\label{eq:27}
p_{t}=-\frac{p_{xxx}}{2}+3pp_{x}.
\end{equation}
Moreover, the converse is true and any solution of \eqref{eq:27} gives
rise, at least locally, to $f:\Sigma\times I\to \R P^{1}$, unique up
to the action of $\rPSL(2,\R)$, solving \eqref{eq:26}.

The key point now is that this flow of isothermic curves commutes
with the transformation theory of
\MySecs\ref{sec:isoth-subm}--\ref{sec:bianchi-perm-self} and so provides
symmetries of the KdV equation.  For this, we extend the connections
$\nabla^{m}=\D+m\eta$ on $\ul\R^{2}_{\Sigma}$ in the $t$-direction to
get connections $\nabla^{m}$ on $\ul\R^{2}_{\Sigma\times I}$ by
\begin{align*}
\nabla^m_{\partial/\partial t}\psi &= -\frac{p_{x}}{2}\psi +
(p-2m)\psi_x\\
\nabla^m_{\partial/\partial t}\psi_{x} &= \bigl(-\frac{p_{xx}}{2}+(p-2m)(p+m)\bigr)\psi+\frac{p_x}{2}\psi_{x}.
\end{align*}
One readily checks that the connections $\nabla^{m}$ are flat for all
$m\in\R$ exactly when \eqref{eq:27} holds.  In fact, these connections
give the AKNS zero-curvature formulation of KdV (see, for example
\cite{Fordy1994}), albeit in a less familiar gauge.

Now fix $\hat m\in\R^{\times}$ and let $\hat f$ be a $\nabla^{\hat
m}$-parallel complement to $f$:
\begin{equation*}
\ul\R^2_{\Sigma\times I}=f\oplus \hat f.
\end{equation*}
In particular, for each $t\in I$, $\hat
f_{|\Sigma\times\set{t}}=\darb_{\hat m}f_{|\Sigma\times\set{t}}$.  To
analyse $\hat f$, we choose $\hat\psi\in\Gamma\hat f$ with
$\psi\wedge\hat\psi=\psi\wedge\psi_{x}$ so that
\begin{equation}
\label{eq:28}
\hat\psi=a\psi+\psi_{x}.
\end{equation}
One computes that $\hat f$ is $\nabla^{\hat m}$-parallel if and only
if
\begin{subequations}
\label{eq:29}
\begin{align}
\label{eq:30}
a^2-a_x-\hat m&=p\\\label{eq:31}
a_t-\frac{p_{xx}}{2}+p^{2}-ap_{x}-a^{2}p&=-\hat m(2a^{2}-2\hat m-p)
\end{align}
\end{subequations}
We recognise \eqref{eq:30} as the Miura transform and from it deduce
first that $\hat\psi$ is a normalised section of $\hat f$:
$\hat\psi\wedge\hat\psi_x=\hat m\psi\wedge\psi_x$ is constant; then
that the projective curvature $\hat p$ of $\hat f$ is given by
\begin{equation*}
\hat p=a^2+a_x-\hat m=p+2a_x.
\end{equation*}
Moreover, it is not difficult to check that
\begin{align*}
\hat\psi_x&\equiv -\hat m \psi\mod \hat f\\
\hat\psi_t&\equiv -\hat m \hat p \psi\mod \hat f
\end{align*}
so that $\hat f_t=\hat p\hat f_x$ whence $\hat p$ is a new solution of
the KdV equation \eqref{eq:27}.  Thus $\hat f$ also gives rise to a
family $\hat\nabla^{m}$ of flat connections on $\ul\R^2_{\Sigma\times I}$ and
one can verify that the identity of Theorem~\ref{th:4} holds in this
extended context:
\begin{equation*}
\Gamma_f^{\hat f}(1-m/\hat m)\cdot\nabla^{m}=\nabla^{\hat m}
\end{equation*}
and then argue exactly as in \MySec\ref{sec:bianchi-perm-darb} to
establish Bianchi permutability of our extended Darboux
transformations.

Of course, this transformation of KdV solutions is not new: it is the
B\"acklund transformation discovered by Wahlquist--Estabrook
\cite{WahlquistEstabrook1971} who used precisely the system
\eqref{eq:29}. Following \cite{WahlquistEstabrook1971}, we can
eliminate $p$ from \eqref{eq:29} and arrive at an mKdV equation for
$a$:
\begin{equation}
\label{eq:33}
a_t=-\frac{a_{xxx}}{2}+3(a^2-\hat m)a_x.
\end{equation}
Conversely, any solution $a$ of \eqref{eq:33} gives rise to a
B\"acklund pair $(p,\hat p)$ of KdV solutions.  All this also admits
a geometric interpretation: consider the map $\phi=(f,\hat
f):\Sigma\times I\to Z=\R P^1\times \R P^1\setminus\Delta$ into the
symmetric space of complementary pairs---a space-form with indefinite
metric.  For fixed $t$, $\phi$ is a curve of constant velocity,
$(\phi_{x},\phi_x)=-4\hat m$ and curvature $\kappa$ given by
\begin{equation*}
\kappa=a/\sqrt{\lvert m\rvert}
\end{equation*}
so that the mKdV equation is also an equation on a curvature.

Taken as a whole, $\phi$ evolves by 
\begin{equation}
\label{eq:24}
\phi_t=(a^2-\hat m)\phi_x-2 \sqrt{\lvert\hat m\rvert}a_x n,
\end{equation}
for $n$ a positively oriented unit vector field orthogonal to
$\phi_x$.  This may be viewed as a higher flow of the curved flat
system and any solution has $a$ solving the mKdV equation
\eqref{eq:33}.  Conversely, a solution of \eqref{eq:33} locally
determines $\phi$ solving \eqref{eq:24} and then solutions $(f,\hat
f)$ of \eqref{eq:26}.

To summarise: the KdV and mKdV flows correspond, via appropriate
curvatures, to flows on isothermic curves and curved flats
respectively and then the Miura transform relating the curvatures
amounts to projecting the curved flats onto the isothermic curves.

A similar analysis is available in higher dimensions starting from
the observation that the Davey--Stewartson equations amount to a
flow on conformal immersions of a surface in $S^4$ which preserves
the class of isothermic immersions \cite{Burstall2002b}.  We shall
return to this elsewhere.

\subsubsection{Curved flats in symmetric $R$-spaces}
\label{sec:curv-flats-riem}

Recall that if $\ci\in\Aut(\fg)$ is the Cartan involution of the
compact real form $K$, then we have a $K$-equivariant inclusion
$\iota_{\ci}:M\to Z$ into the associated symmetric space of
complementary pairs given by $\fp\mapsto(\fp,\ci\fp)$.  The solder
forms of $M$ and $Z$ are related by
$\beta^M=\iota_{\ci}^{*}\beta^{Z}$ so that if $f:\Sigma\to M$ is a
curved flat in $M$, where the latter is viewed as a Riemannian
symmetric $K$-space, then $\iota_{\ci}\circ f=(f,\ci f):\Sigma\to Z$
is a curved flat.  Now Theorem~\ref{th:19} applies and we conclude
that $f:\Sigma\to M$ is isothermic and $\ci f$ is a Darboux transform
of $f$.

Moreover, it is often the case that the rank of $M$, \emph{qua}
Riemannian symmetric $K$-space, coincides\footnote{Indeed, among the
simple symmetric $R$-spaces, $\rank M=\rank Z$ for all the non-complex
examples except the conformal sphere $\pr(\cL^{n+1,1})$, the Cayley
plane, the quaternionic Grassmannians and the quaternionic Lagrangian
Grassmannian.} with that of $Z$: $\rank M=\rank Z$.  In this
situation, curved flats of maximal rank in $M$ are non-degenerate
isothermic submanifolds.

Thus we obtain, for example, isothermic submanifolds of the
Grassmannian $G_{k+1}(\R^{n+1})$ from the Gauss maps of isometric
immersions of certain $k$-dimensional space-forms into $S^{n}$
\cite{Ferus1996} and isothermic submanifolds of the Lagrangian
Grassmannian $\Lag(\R^{2n})$ from Egoroff nets and the Gauss maps of
flat Lagrangian submanifolds in $\C^{n}$ and $\C\pr^{n-1}$
\cite{Terng2008}.

In fact, it is not even necessary to require that $\ci$ is a Cartan
involution: all that is needed is to restrict attention to the
open subset $\Omega_{\ci}=\set{\fp\in M: (\fp,\ci\fp)\in
Z}$, for an arbitrary involution $\ci\in\Aut(\fg)$.  For example, the
Grassmannians $G_{k,l}(\R^{p,q})\subset G_{k+l}(\R^{p+q})$ of
signature $(k,l)$ subspaces of $\R^{p,q}$ are of this kind and the many
examples of curved flats in these Grassmannians
\cite{Bruck2002,Burstall2004,Burstall2004c,Ferus1996} are all
isothermic submanifolds of the real Grassmannian $G_{k+l}(\R^{p+q})$.

\appendix

\section{Summary of simple symmetric $R$-spaces}
\label{sec:summ-simple-symm}

Table~\ref{tab:r-space} lists the symmetric $R$-spaces for simple
$G$.  For each such space $M$, we give the group $G$; the realisation
of $M$ as a Riemannian symmetric $K$-space (recall that $K$ is a
maximal compact subgroup of $G$); the dimension of $M$; whether $M$
is self-dual and $\rank Z$, the maximal dimension of a nondegenerate
isothermic submanifold of $M$.
\begin{table}[h]
\begin{equation*}
\begin{array}{llllll}
   M & G & \text{$K$-space} &
   \dim M & M=M^{*} & \rank Z\\[4pt]
G_p(\C^{p+q}) & \rSL(p+q,\C) & \frac{\rSU(p+q)}{\rSU(p)\times\rSU(q)\times S^1} & 2pq & p=q & 2\min(p,q)\\
G_p(\R^{p+q}) & \rSL(p+q) & \frac{\rSO(p+q)}{\rSO(p)\times\rSO(q)\times\Z_2} & pq & p=q & \min(p,q)\\
G_p(\HH^{p+q}) & \rSU^*(2p+2q) &\frac{\rSp(p+q)}{\rSp(p)\times\rSp(q)} & 4pq & p= q &2\min(p,q)\\
G^{\perp}_n(\C^{n,n}) & \rSU(n,n) & \frac{\rSU(n)\times\rSU(n)\times S^1}{\rSU(n)}=\rU(n) & n^2 &\text{Yes}& n\\[8pt]
\pr(\cL^{\C})& \rSO(n+2,\C) & \frac{\rSO(n+2)}{\rSO(n)\times\rSO(2)\times\Z_2} & 2n &\text{Yes}& 4\\
\pr(\cL^{p+1,q+1})& \rSO(p+1,q+1) &
\frac{\rSO(p+1)\times\rSO(q+1)}{\rSO(p)\times\rSO(q)\times\Z_2} & p+q&\text{Yes} & 2\\[8pt]
J^{\pm}(\C^{2n}) & \rSO(2n,\C) & \rSO(2n)/\rU(n) & n(n-1) & \text{$n$ even} & 2\lfloor n/2\rfloor\\
J^{\pm}(\R^{n,n}) & \rSO(n,n) &\frac{\rSO(n)\times\rSO(n)}{\rSO(n)}=\rSO(n) & \frac{1}{2}n(n-1)&\text{$n$ even} & \lfloor n/2\rfloor\\
J(\HH^{2n}) & \rSO^*(4n) & \rU(2n)/\rSp(n) & n(2n-1) &\text{Yes} & n\\[8pt]
\Lag(\C^{2n}) & \rSp(n,\C) & \rSp(n)/\rU(n) & n(n+1) &\text{Yes} & 2n\\
\Lag(\R^{2n}) & \rSp(n,\R) & \rU(n)/\rSO(n) & \frac{1}{2}n(n+1) &\text{Yes}& n\\
\Lag(\HH^{2n}) & \rSp(n,n) & \frac{\rSp(n)\times\rSp(n)}{\rSp(n)}=\rSp(n) & n(2n+1) &\text{Yes}& 2n\\[8pt]
 & E_6^\C & E_6/\text{Spin }(10)\times S^1 & 32 &\text{No}& 4\\
 & E_6 I & \rSp(4)/\rSp(2)\times\rSp(2) & 16 &\text{No}& 2\\
\pr^{2}(\mathbb{O}) & E_6 IV & F_4/\text{Spin }(9) & 16 &\text{No}& 2\\[8pt]
 & E_7^\C & E_7/E_6\times S^1 & 54 &\text{Yes}& 6\\
 & E_7 V & \rSU(8)/\rSp(4) & 27 &\text{Yes} & 3\\
 & E_7 VII & E_6\times S^1/F_4 & 27 &\text{Yes}& 3
\end{array}
\end{equation*}
\caption{The symmetric $R$-spaces of simple type}
\label{tab:r-space}
\end{table}

\newcommand{\noopsort}[1]{}
\providecommand{\bysame}{\leavevmode\hbox to3em{\hrulefill}\thinspace}
\providecommand{\MR}{\relax\ifhmode\unskip\space\fi MR }
\makeatletter
\def\@strippedMR{}
\def\@scanforMR#1#2#3\endscan{%
  \ifx#1M\ifx#2R\def\@strippedMR{#3}%
  \else\def\@strippedMR{#1#2#3}%
  \fi\fi}
\renewcommand\MR[1]{\relax
  \ifhmode\unskip\spacefactor3000 \space\fi
  \@scanforMR#1\endscan
  MR\MRhref{\@strippedMR}{\@strippedMR}}
\makeatother 
\providecommand{\MRhref}[2]{%
  \href{http://www.ams.org/mathscinet-getitem?mr=#1}{#2}
}
\providecommand{\href}[2]{#2}

\end{document}